\documentclass[11pt]{amsart}
\usepackage[utf8]{inputenc}
\usepackage{amsmath,amsthm,amssymb}
\newtheorem{thm}{Theorem}[section]
\newtheorem{lemma}[thm]{Lemma}
\newtheorem{remark}[thm]{Remark}
\newtheorem{corollary}[thm]{Corollary}
\theoremstyle{definition}

\newcommand{\abs}[1]{ \left\lvert#1\right\rvert} 
\newcommand{\norm}[1]{\left\lVert#1\right\rVert} 
\numberwithin{equation}{section}
\usepackage{graphicx}
\usepackage[font=tiny,labelfont=bf]{caption}

\title[Pulsatile flow in viscoelastic vessels]
{Well-posedness, ill-posedness, and traveling waves for 
models of pulsatile flow in viscoelastic vessels}
\author{Hyeju Kim}
\address{Department of Mathematics, Drexel University, Philadelphia, PA 19104, USA}
\email{hk655@drexel.edu}
\author{David M. Ambrose}
\address{Department of Mathematics, Drexel University, Philadelphia, PA 19104, USA}
\email{dma68@drexel.edu}

\begin{document}

\begin{abstract} We study dispersive models of fluid flow in viscoelastic vessels, derived in the study of blood flow.  
The unknowns in the models are the velocity of
the fluid in the axial direction and the displacement of the vessel wall from rest.  We prove that one such model has a well-posed
initial value problem, while we argue that a related model instead has an ill-posed initial value problem; in the second case,
we still prove the existence of solutions in analytic function spaces.  Finally we prove the existence of some periodic traveling
waves.
\end{abstract}

\maketitle

\section{Introduction}

We consider a fluid-structure interaction problem with a fluid flowing within a viscoelastic vessel, motivated by hemodynamics.  
The specific models to be studied have been derived in
\cite{mitsotakis1}, based on the prior work \cite{mitsotakis2}.  As shown in Figure \ref{fig:vessel}, we consider an axisymmetric flow.
The model equations begin from the Navier-Stokes equations for incompressible flow, making a number of assumptions, such
as laminar flow with small viscosity.  

 \begin{figure}[htp]
    \centering
    \includegraphics[width=10cm]{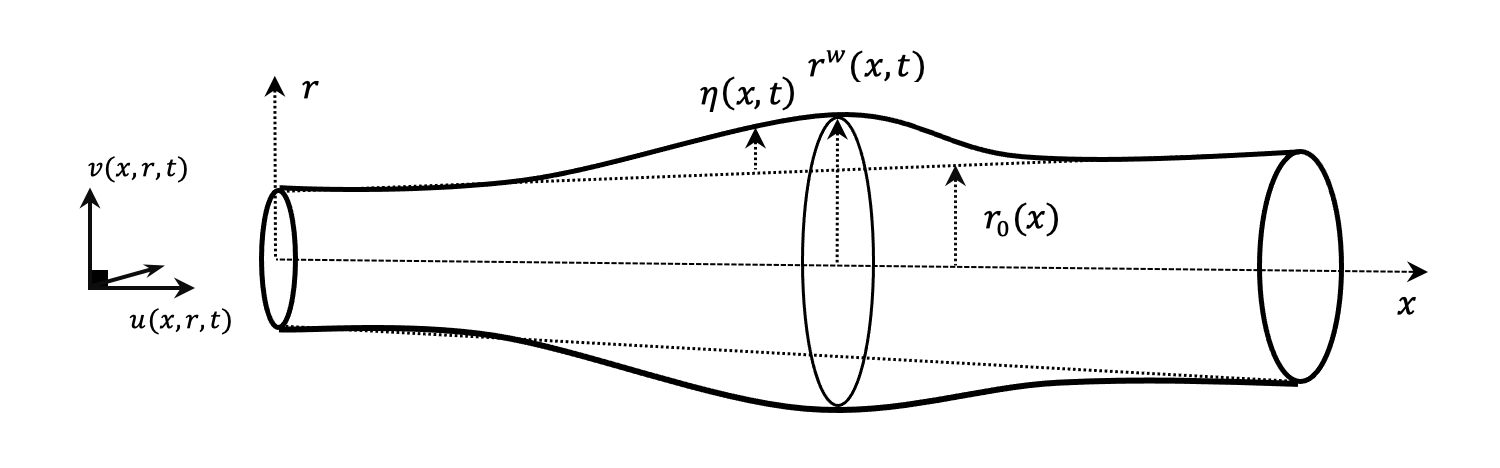}
    \caption{Sketch of a vessel segment.} 
    \label{fig:vessel}
\end{figure}

The vessel containing the fluid is taken to have a given undisturbed radius $r_{0}(x),$ and we study the displacement, $\eta(x,t),$
of this; we call the total radius of the vessel, then, $r^{w}(x,t)=r_{0}(x)+\eta(x,t).$  The horizontal component of the fluid velocity
is $u(x,t);$ this is taken to be the horizontal velocity at a particular distance between the centerline of the vessel and the outer
wall.  A classical Boussinesq system of equations is derived, making assumptions on the scaling of the various velocities; this system
is
\begin{equation}\label{originalEtaEquation}
\eta_{t}+\frac{1}{2}(r_{0}+\eta)u_{x}+(r_{0}+\eta)_{x}u=0,
\end{equation}
\begin{multline}\label{originalUEquation}
[1-\bar{\alpha}r_{0xx}]u_{t}+(\bar{\beta}\eta)_{x}+uu_{x}-\frac{(4\bar{\alpha}+r_{0})r_{0}}{8}u_{xxt}
+\frac{(3\bar{\alpha}+r_{0})r_{0x}}{2}(\bar{\beta}\eta)_{xx}
\\
+\kappa u-\gamma\left(\bar{\beta}(r_{0x}u+\frac{r_{0}}{2}u_{x})_{x}\right)
=0.
\end{multline}
There are a number of parameters here which must be described.  First, $\rho$ is the density of the fluid while $\rho^{w}$ is
the density of the wall material, and $h$ is the thickness of the wall.  These are combined in the parameter 
$\bar{\alpha}=\frac{\rho^{w}h}{\rho},$ measuring the relative densities of the wall and the fluid.  The parameter $E$ measures
elasticity of the wall, and then $\bar{\beta}$ (which is a function of $x$ rather than being constant), is given by 
$\bar{\beta}(x)=\frac{Eh}{\rho r_{0}^{2}(x)}.$    The parameters $\kappa$ and $\gamma$ are both viscosities, with $\kappa$
being the fluid viscous frequency parameter (i.e., the Rayleigh damping coefficient).  We have said that the wall of the vessel
is taken to be viscoelastic, and $\gamma$ measures the viscous properties of the wall.

Prior models have considered the vessel wall to be elastic, rather than viscoelastic \cite{cascaval}, \cite{mitsotakis2}.  
However, accurate modeling of the anatomy of blood vessels requires the more detailed (viscoelastic) description.
Specifically, as described in \cite{bertagliaCaleffiValiani}, there are three layers of a blood vessel, the tunica intima (inner layer),
tunica media (middle layer), and tunica externa (outer layer), and the smooth muscle cells in the tunica media exhibit viscoelastic
properties \cite{banks}.  Furthermore, in some regimes, these viscoelastic properties are dominant as compared to purely
elastic effects \cite{jengmath}.

In the case that $r_{0}$ is constant, the model 
simplifies considerably; notice that not only do derivatives of $r_{0}$ now vanish, but also
$\bar{\beta}$ becomes constant so that its derivatives now also vanish.  The result is
\begin{equation}\label{etaEquation}
\eta_{t}+\frac{1}{2}(r_{0}+\eta)u_{x}+\eta_{x}u=0,
\end{equation}
\begin{equation}\label{uEquation}
u_{t}+\bar{\beta}\eta_{x}+uu_{x}-\frac{(4\bar{\alpha}+r_{0})r_{0}}{8}u_{xxt}
+\kappa u-\gamma\left(\bar{\beta}\frac{r_{0}}{2}u_{xx}\right)
=0.
\end{equation}

We prove three main results in the present work.  First, for the model \eqref{etaEquation}, \eqref{uEquation} with constant $r_{0},$
we demonstrate well-posedness of the initial value problem in Sobolev spaces.  Notably, by contrast, we provide evidence that
the more general model \eqref{originalEtaEquation}, \eqref{originalUEquation} instead has an ill-posed initial value problem.
That an initial value problem is ill-posed does not imply that there are no solutions, however.  An example of this is the classical
vortex sheet initial value problem, which is known to be ill-posed in Sobolev spaces \cite{caflischOrellana}.
Existence of solutions for the vortex sheet problem may be established in analytic function spaces \cite{duchonRobert}, 
\cite{sulemSulemBardosFrisch}.  Similarly to \cite{sulemSulemBardosFrisch}, we prove existence of solutions for 
the initial value problem for \eqref{originalEtaEquation}, \eqref{originalUEquation} in analytic spaces based on the Wiener algebra,
making use of an abstract Cauchy-Kowalevski theorem \cite{kanoNishida}.
The interested reader might also see \cite{rafaelJon}, \cite{rafaelOddViscosity} for other examples of model equations in free-surface
fluid dynamics for which solutions have been proved to exist in analytic function spaces, when the well-posedness in spaces
of finite regularity is in question.

The model \eqref{originalEtaEquation}, \eqref{originalUEquation} is bidirectional, in that waves may propagate either to the left 
or the right.  The authors of \cite{mitsotakis1} also derive unidirectional models, related to the Korteweg-de Vries equation
and the Benjamin-Bona-Mahony equation.  These models are simpler, and reduce to a single equation for $\eta.$
We consider the contrast in the bidirectional case between well-posedness when $r_{0}$ is constant and likely ill-posedeness
when $r_{0}$ is non-constant to be an interesting feature of the present work; this constrast is not present in the unidirectional
models, as (relying on results such as those of \cite{akhunov}, \cite{nonl4}, \cite{IUMJ3}, or \cite{craig}) the unidirectional models
can be shown to be well-posed in either case.  As the bidirectional models are therefore more interesting, we restrict our
studies to them.

In addition to developing the models we study here, the authors of \cite{mitsotakis1} also studied properties of traveling waves,
including the case $\gamma=\kappa=0.$   For our third main result, then, 
we prove existence of such waves.  Specifically, we prove existence of 
periodic traveling waves of the system \eqref{etaEquation}, \eqref{uEquation}
in the case that $\gamma=\kappa=0.$  This is the doubly inviscid case, meaning that for the existence of traveling waves,
we neglect the viscous properties of the fluid and of the vessel wall.  We prove this by a ``bifurcation from a simple eigenvalue''
method, after studying the kernel of the linearized operator associated to \eqref{etaEquation}, \eqref{uEquation}.  In general
this operator has a two-dimensional kernel, but when $\gamma=\kappa=0,$ we may enforce symmetry, reducing the dimension 
of the kernel to one.  Analytical studies of the traveling waves in the more general case, with the two-dimensional kernel, 
will be the subject of future work.

The plan of the paper is as follows.  In Section \ref{wellPosednessSection} we prove well-posedness in Sobolev spaces 
of the initial value problem for the system \eqref{etaEquation}, \eqref{uEquation}; the main theorems of this section are
Theorem \ref{existenceTheorem} demonstrating existence, and Theorem \ref{CDOICTheorem} demonstrating uniqueness
and continuous dependence on the data.  In Section \ref{ACKSection} we 
give a calculation suggesting ill-posedness of the more general system \eqref{originalEtaEquation}, \eqref{originalUEquation},
and then prove existence of solutions for this system in analytic function spaces by application of an abstract Cauchy-Kowalevski
theorem.  The main theorem of Section \ref{ACKSection} is Theorem \ref{secondMain}.  In Section \ref{travelingSection} 
we prove existence of periodic traveling waves for the system \eqref{etaEquation}, \eqref{uEquation}
when $\kappa=\gamma=0;$ this is the content of Theorem \ref{travelingWaveExistenceTheorem}.  
We make some concluding remarks in Section \ref{discussionSection}.

\section{Well-Posedness in Sobolev Spaces when $r(x)$ is constant} \label{wellPosednessSection}

In this section we use the energy method to prove well-posedness in Sobolev spaces of the
spatially periodic initial value problem for the
system \eqref{etaEquation}, \eqref{uEquation}.  We argue along the same lines as the second other used for a toy model for
the vortex sheet with surface tension in \cite{ambroseCambridgeLecture}.  

We recall the model \eqref{etaEquation}, \eqref{uEquation}, and we rearrange terms as follows:
   \begin{equation}\nonumber
      \eta_t = - \frac{1}{2}r_0u_x - \frac{1}{2}\eta u_x - \eta_xu,
      \end{equation}
      \begin{equation}\nonumber
       u_t = \left(1-\frac{(4\overline{\alpha}
      +r_0)r_0}{8}\partial_{xx}\right)^{-1}
      \left(- \overline{\beta}\eta_x - uu_x  - \kappa u + \frac{\gamma \overline{\beta} r_0}{2}u_{xx}\right).
   \end{equation}
We introduce an approximate system, giving equations for 
$\eta^{\epsilon}_t$ and $u^{\epsilon}_t$ using mollifier operators  
$\mathcal{J_\epsilon}$ for any approximation parameter $\epsilon > 0.$  (For a detailed description of mollifier operators and
their properties, the interested reader could consult Chapter 3 of \cite{majdaBertozzi}; it is enough to say that they are 
self-adjoint smoothing operators, and could be taken specifically to be truncation of the Fourier series at level $1/\varepsilon.$)
Our approximate system is:
\begin{eqnarray}
    &&\eta^{\epsilon}_{t}       \label{etaEpsilonT}
                       =- \frac{1}{2}r_0u^{\epsilon}_x - \frac{1}{2}\eta^{\epsilon}  u^{\epsilon}_x 
                      - \mathcal{J}_{\epsilon}\left( (\mathcal{J}_{\epsilon}\eta^{\epsilon}_x) u^{\epsilon}\right),\\
    &&u^{\epsilon}_{t}           \label{uEpsilonT}
                    =A^{-1}\left(- \overline{\beta}\eta^\epsilon_x - u^\epsilon u^\epsilon_x  - \kappa u^{\epsilon} + \frac{\gamma \overline{\beta} r_0}{2}u^\epsilon_{xx}\right),
\end{eqnarray}
where 
$A^{-1} =\left[1-\frac{(4\overline{\alpha}
+r_0)r_0}{8}\partial_{xx}\right]^{-1}.$
The system \eqref{etaEpsilonT}, \eqref{uEpsilonT} is taken with 
initial conditions, namely
   \begin{equation}\label{initialConditionsEpsilon}
       \eta^{\epsilon}(\cdot,0) = \eta_{0}\in H^{s},\qquad  u^{\epsilon}(\cdot,0) = u_{0}\in H^{s+1}.
   \end{equation}
 Here, $s\in\mathbb{N}$ with $s\geq 2,$ and
 $H^{s}=H^{s}(\mathbb{T})$ and $H^{s+1}=H^{s+1}(\mathbb{T})$ are the standard 
 spatially periodic $L^{2}$-based Sobolev spaces, equipped with the usual norms.

We will show that given initial data $\eta_0$ and $u_0$, there exists a time interval $[0,T]$ (depending only on the size
of the data) such that there exists a solution $(\eta,u)$ solving our initial value problem over the time interval $[0,T].$
Our first step is to apply the Picard Theorem on 
Banach spaces, which we now state \cite{majdaBertozzi}.

   \begin{thm}[Picard Theorem]
       Let $\mathcal{B}$ be a Banach space, and let $\mathcal{O}\subseteq \mathcal{B}$ be an open set.  Let $F:\mathcal{O} \rightarrow \mathcal{B}$ such that $F$ is locally Lipschitz: $\forall X \in \mathcal{O}, \exists \lambda > 0$ and an open set $U \subseteq \mathcal{O}$ such that $\forall Y, Z \in U$,
      \[\norm{F(Y)-F(Z)}_{\mathcal{B}} \leq \lambda \norm{Y-Z}_{\mathcal{B}}.\]
      Then,  $\forall X_0 \in \mathcal{O}, \exists\ T>0$ and a unique $X \in C^{1}([-T,T];\mathcal{O})$ such that $X$ solves the initial value problem 
      \[ \frac{dX}{dt} = F(X),\quad X(0)= X_0. \]
   \end{thm}

We will take $\mathcal{O} = \mathcal{B} = H^{s} \times H^{s+1}$ and introduce the following lemma:
\begin{lemma}\label{lipschitz}
      Let $(\eta_0,u_0) \in \mathcal{O}$ be given. For any $\epsilon > 0$, there exists  $T_\epsilon > 0$ and $(\eta^{\epsilon}, u^{\epsilon}) \in C^{1}([0,T_{\epsilon}]; \mathcal{O})$ such that $(\eta^{\epsilon}, u^{\epsilon})$ satisfies 
      \eqref{etaEpsilonT}, \eqref{uEpsilonT}  and the initial conditions
      \eqref{initialConditionsEpsilon}.
   \end{lemma}

We omit the proof of Lemma \ref{lipschitz}; it follows immediately from the Picard Theorem and from properties of mollifiers.   
Note that we only introduced two mollifier operators on the right-hand side of \eqref{etaEpsilonT}, and none on the right-hand
side of \eqref{uEpsilonT}.  For \eqref{etaEpsilonT}, this is because when solving \eqref{etaEquation} for $\eta_{t},$ 
if we consider $(\eta,u)\in H^{s}\times H^{s+1},$ then the only unbounded term is $\eta_{x}u.$  
(We have included two instances of $\mathcal{J}_{\epsilon}$ 
to be able to achieve a balance when integrating by parts in the energy estimates to follow.)
For \eqref{uEpsilonT},
when solving \eqref{uEquation} for $u_{t}$ and again considering $(\eta,u)\in H^{s}\times H^{s+1},$ there are no unbounded terms
(because of the presence of the operator $A^{-1}$).

\subsection{Energy Estimate}
Next, we will show that there exists $T > 0$ and $\epsilon_0 >0$ such that for all $\epsilon \in (0,\epsilon_0)$, the solutions $(\eta^\epsilon,u^\epsilon)$ are elements of $C([0,T];\mathcal{O})$. In order to complete the proof, we will use the
 following ODE theorem \cite{majdaBertozzi}:

   \begin{thm}[Continuation Theorem for ODEs]\label{continuationTheorem}
      Let $\mathcal{B}$ be a Banach space and $\Omega \subseteq \mathcal{B}$ be an open set and $F: \Omega \rightarrow{\mathcal{B}}$ be locally Lipschitz continuous.
      Let $X_0=(\eta_0,u_0) \in \Omega$ and $X=(\eta,u)$ be the solution of initial value problem:\\
      \begin{equation*}
         \frac{dX}{dt}= F(X),\quad X(0)= X_0,\\
      \end{equation*}
      and let $T > 0$ be the maximal time such that $X \in C^{1}([0,T];\Omega)$. Then either $T = \infty$ or $T < \infty$ with $X(t)$ leaving the set $\Omega$ as $t \rightarrow{T}$. 
     
   \end{thm}

In order to use Theorem \ref{continuationTheorem}, we need to prove that 
the norm of $(\eta^{\epsilon},u^{\epsilon})$ may be controlled uniformly with respect to $\epsilon$. 
We establish this in the following lemma using the energy method.

   \begin{lemma}\label{uniformTime}
      Let $(\eta_0,u_0) \in \mathcal{O}$. There exists $T > 0$ such that for all  $\epsilon \in (0,1]$,
      the initial value problem  \eqref{etaEpsilonT}, \eqref{uEpsilonT}, \eqref{initialConditionsEpsilon} 
      has a solution $(\eta^\epsilon, u^\epsilon) \in C([0,T],\mathcal{O})$.
   \end{lemma}

   \begin{proof}
   Let $\epsilon \in (0,1]$ be given. We know there exists $T_{\epsilon}>0$ and 
   $(\eta^{\epsilon}, u^{\epsilon}) \in C^{1}([-T_{\epsilon},T_{\epsilon}];\mathcal{O})$, which solves the regularized initial value
   problem.
   Now, we will show that these solutions can be continued until a time $T$, with $T$ being independent of $\epsilon$.
   
We define an energy $E(t) = E_0(t) + E_1(t) + E_2(t)$ to be
      \begin{eqnarray}
\nonumber         E_0(t) &=& \frac{1}{2} \int_0^{2\pi} (\eta^\epsilon)^2 +(u^\epsilon)^2 dx,\\ 
\nonumber         E_1(t) &=& \frac{1}{2} \int_0^{2\pi} (\partial^{s}_x  \eta^\epsilon)^2 dx,\\
\nonumber        E_2(t) &=& \frac{1}{2} \int_0^{2\pi} (\partial^{s+1}_x  u^{\epsilon})^2 dx.
      \end{eqnarray}
Of course, this energy is equivalent to the square of the $H^{s}$-norm of $\eta^{\epsilon}$ 
plus the square of the $H^{s+1}$-norm of $u^{\epsilon}.$
We will show that the time derivative of the energy is bounded in terms of the energy, as long as $s\geq2.$  

We begin with showing $\frac{dE_0}{dt}$ is bounded appropriately, so we calculate
\begin{equation}\label{startingEnergyT}
    \frac{dE_0}{dt} 
    =\int_0^{2\pi} \eta^\epsilon(\eta^\epsilon_t) + u^\epsilon (u^\epsilon_t)\ dx.
\end{equation}
Substituting \eqref{etaEpsilonT} and \eqref{uEpsilonT} into \eqref{startingEnergyT}, we have
   \begin{align*}
       \frac{dE_0}{dt}
        =&\int_0^{2\pi} \eta^{\epsilon} \left(- \frac{1}{2}r_0u^{\epsilon}_x - \frac{1}{2}\eta^{\epsilon}  u^{\epsilon}_x 
        -\mathcal{J}_{\epsilon}\left( (\mathcal{J}_{\epsilon}\eta^{\epsilon}_x) u^{\epsilon}\right)\right)\ dx \\
         &+ \int_0^{2\pi}u^{\epsilon}\left(1-\frac{(4\overline{\alpha}+r_0)r_0}{8}\partial_{xx}\right)^{-1}\left(- \overline{\beta}\eta_x - uu_x  - \kappa u + \frac{\gamma \overline{\beta} r_0}{2}u_{xx}\right)\ dx.
   \end{align*}
We may then immediately bound this as
\begin{multline}\nonumber
\frac{dE_0}{dt} 
\leq c\bigg(\ \norm{\eta^{\epsilon}}_{H^0} \norm{u^{\epsilon}}_{H^1} 
  +\norm{\eta^{\epsilon}}_{H^0}  \norm{\eta^{\epsilon}}_{H^0} \norm{u^{\epsilon}}_{H^2}\\
  +\norm{\eta^\epsilon}_{H^0} \norm{\eta^{\epsilon}}_{H^2}   
  \norm{u^\epsilon}_{H^0}
  +\norm{u^{\epsilon}}_{H^0} \norm{\eta^{\epsilon}}_{H^0}
  +\norm{u^{\epsilon}}^{3}_{H^0}
  +\norm{u^{\epsilon}}^{2}_{H^0} +\norm{u^{\epsilon}}^{2}_{H^0}\bigg).
\end{multline}
Therefore, $\frac{dE_0}{dt}$ satisfies the following energy estimate as long as $s\geq 2$:
\begin{equation}\nonumber
    \frac{dE_0}{dt} \leq c(E + E^{\frac{3}{2}}).
\end{equation}

Now, we turn to $E_{1};$ taking its time derivative, we have
\begin{equation}\label{dE_1}
    \frac{dE_1}{dt}
  =\int_0^{2\pi} (\partial^{s}_x \eta^\epsilon)(\partial^{s}_x \eta^{\epsilon}_t)\ dx.
\end{equation}
Substituting \eqref{etaEpsilonT}  into \eqref{dE_1}, we have
\begin{multline}\label{dE1dt}
    \frac{dE_1}{dt}
     =-\frac{r_0}{2} \int_0^{2\pi} (\partial^{s}_x  \eta^\epsilon)(\partial^{s+1}_x u^\epsilon)\ dx 
    - \frac{1}{2} \int_0^{2\pi} (\partial^{s}_x \eta^\epsilon)\partial^{s}_x[\eta^\epsilon u^{\epsilon}_x]\ dx \\
    + \int_0^{2\pi} (\partial^{s}_x \mathcal{J}_{\epsilon} \eta^\epsilon) \partial^{s}_x [(\mathcal{J}_\epsilon \eta^{\epsilon}_x)
    u^{\epsilon}]\ dx
    =\sum_{k=1}^{3} \Psi_k.
\end{multline}
In the formula for $\Psi_{3},$ we have already used that the mollifier operator $\mathcal{J}_{\epsilon}$ is self-adjoint.
We will  show each $\Psi_k$ in \eqref{dE1dt} is bounded in terms of the energy, $E$.

Since the energy is equivalent to the sum of the square of the $H^{s}$-norm of $\eta^{\epsilon}$ and the square of the
$H^{s+1}$-norm of $u^{\epsilon},$
the bound
\begin{equation}\label{psi1Conclusion}
\Psi_1 \leq c E
\end{equation}
is immediate.  
For $\Psi_{2},$ we immediately may bound it as
\begin{equation}\nonumber
\Psi_{2}\leq\|\partial_{x}^{s}\eta^{\epsilon}\|_{0}
\|\partial_{x}^{s}\left(\eta^{\epsilon} u^{\epsilon}_{x}\right)\|_{0}
\leq\|\eta^{\epsilon}\|_{s}\|\eta^{\epsilon} u^{\epsilon}_{x}\|_{s}.
\end{equation}
Since $s\geq1,$ we may use the Sobolev algebra property, finding
\begin{equation}\label{psi2Conclusion}
    \Psi_2 \leq c\|\eta^{\epsilon}\|_{s}^{2}\|u^{\epsilon}\|_{s+1}\leq c E^{\frac{3}{2}}.
\end{equation}

Now, we turn to the third term, $\Psi_{3},$ on the right-hand side of \eqref{dE1dt}. Using the product rule to expand
derivatives, $\Psi_{3}$ can be rewritten as follows:
\begin{equation}\label{psi3Expansion}
    \Psi_3
    =-\int_0^{2\pi} (\partial^{s}_x\mathcal{J}_\epsilon \eta^\epsilon)  \sum_{k=0}^{s} \binom{s}{k} (\partial^{k+1}_x \mathcal{J}_\epsilon \eta^{\epsilon})(\partial^{s-k}_x  u^\epsilon)\ dx.
\end{equation}
The most singular term on the right-hand side of \eqref{psi3Expansion} is the $k=s$ term, for which all derivatives fall on 
$\eta^{\epsilon}.$  Thus we decompose \eqref{psi3Expansion} as
\begin{multline}\label{psi3MostImportant}
\Psi_{3}= -\int_0^{2\pi}(\partial^{s}_x \mathcal{J}_{\epsilon} \eta^\epsilon) (\partial^{s+1}_x \mathcal{J}_\epsilon \eta^\epsilon)(\partial_x  u^\epsilon)\ dx
\\
-\int_0^{2\pi} (\partial^{s}_x\mathcal{J}_\epsilon \eta^\epsilon)  \sum_{k=0}^{s-1} \binom{s}{k} (\partial^{k+1}_x \mathcal{J}_\epsilon \eta^{\epsilon})(\partial^{s-k}_x  u^\epsilon)\ dx.
\end{multline}
The first term on the right-hand side of \eqref{psi3MostImportant} can be integrated by parts, arriving at
\begin{multline}\label{psi3Final}
\Psi_{3}=
\frac{1}{2}\int_0^{2\pi}(\partial^{s}_x \mathcal{J}_{\epsilon} \eta^\epsilon)^{2} (\partial_{x}^{2}  u^\epsilon)\ dx
\\
-\int_0^{2\pi} (\partial^{s}_x\mathcal{J}_\epsilon \eta^\epsilon)  \sum_{k=0}^{s-1} \binom{s}{k} (\partial^{k+1}_x \mathcal{J}_\epsilon \eta^{\epsilon})(\partial^{s-k}_x  u^\epsilon)\ dx.
\end{multline}
We see then that the right-hand side of \eqref{psi3Final} involves at most $s$ derivatives of $\eta^{\varepsilon}$ and at most
$s+1$ derivatives of $u^{\varepsilon};$ this implies 
\begin{equation}\label{psi3Conclusion}
\Psi_{3}\leq cE^{3/2}.
\end{equation}
Combining \eqref{psi1Conclusion}, \eqref{psi2Conclusion}, and \eqref{psi3Conclusion}, we have
\begin{equation}\nonumber
\frac{dE_{1}}{dt}\leq c(E+E^{3/2}).
\end{equation}

Just as $\frac{dE_0}{dt}$ and $\frac{dE_1}{dt}$ are bounded by the energy, we will also show $\frac{dE_2}{dt}$ is bounded by $E$. Taking the derivative of $E_{2}$ with respect to time, we have
 \begin{equation}\label{e2TimeDerivative}
     \frac{dE_2}{dt} 
    = \frac{1}{2} \int_0^{2\pi} (\partial^{s+1}_x  u^{\epsilon}) (\partial^{s+1}_x u^{\epsilon}_t) dx.
 \end{equation}
Substituting  \eqref{uEpsilonT}  into \eqref{e2TimeDerivative} leads to the following sum:
\begin{multline}\label{e2TimeDerivativeSum}
    \frac{dE_2}{dt}
    = -\frac{1}{2} \int_0^{2\pi} \left(\partial^{s+1}_x  u^{\epsilon}\right)
    \left(\partial^{s+1}_x   A^{-1}\overline{\beta}\eta^{\epsilon}_x\right)\ dx\\
    -\frac{1}{2} \int_0^{2\pi} \left(\partial^{s+1}_x  u^{\epsilon}\right)
    \left(\partial^{s+1}_x  A^{-1} (u^\epsilon u^{\epsilon}_x)\right)\ dx 
    -\frac{1}{2} \int_0^{2\pi} \left(\partial^{s+1}_x  u^{\epsilon}\right)
    \left(\partial^{s+1}_x   A^{-1} \kappa u^\epsilon\right)\ dx \\
    +\frac{1}{2} \int_0^{2\pi} \left(\partial^{s+1}_x u^{\epsilon}\right)
    \left(\partial^{s+1}_x   A^{-1}\frac{\gamma \overline{\beta}r_0}{2}u^{\epsilon}_{xx}\right)\ dx
    =\sum_{k=1}^{4} \Omega_k.
\end{multline}

We begin to estimate the first term in the summation in \eqref{e2TimeDerivativeSum}. 
We can bound both factors in $L^{2}:$
\begin{equation}\nonumber
       \Omega_1
    \leq c \norm{\partial^{s+1}_x  u^\epsilon}_{L^2} \norm{ \partial^{s+2}_x  A^{-1} \eta^{\epsilon}}_{L^2}.
\end{equation}
We recall that $A^{-1}$ smoothes by two derivatives, leading us to find
\begin{equation}\nonumber
\Omega_1
    \leq c\norm{u^\epsilon}_{H^{s+1}} \norm{\eta^\epsilon}_{H^s}.
\end{equation}
Thus, we have $\Omega_1$ bounded by the energy:
\begin{equation}\nonumber
      \Omega_1
    \leq c E^{\frac{1}{2}}_2  E^{\frac{1}{2}}_1
    \leq c E.
\end{equation}
Next, we turn to the second summand on the right-hand side of \eqref{e2TimeDerivativeSum}, $\Omega_2.$ 
We again bound each of the two factors in $L^{2}:$
\begin{equation}\nonumber
    \Omega_2
    \leq c \norm{\partial^{s+1}_x  u^{\epsilon}}_{L^2}
    \norm{\partial^{s+1}_x  A^{-1} (u^\epsilon u^{\epsilon}_x)}_{L^2}.
\end{equation}
Again using that $A^{-1}$ smoothes by two derivatives, we have
\begin{equation} \nonumber
 \Omega_2
   \leq c \norm{u^\epsilon}_{H^{s+1}} \norm{u^\epsilon u^{\epsilon}_x}_{H^{s-1}}.
   \end{equation}   
Using the Sobolev algebra property, this yields the desired bound, namely
\begin{equation}\nonumber
\Omega_2
\leq c E^{\frac{3}{2}}.   
\end{equation}
We move on to $\Omega_3,$ and estimate it similarly, finding
\begin{equation*}
    \Omega_3
    \leq c \norm{\partial^{s}_x  u^{\epsilon}}_{L^2} \norm{\partial^{s}_x  u^\epsilon}_{L^2},
\end{equation*}
which implies
   \begin{equation}\nonumber
       \Omega_3
        \leq cE.
   \end{equation}
Lastly, we estimate $\Omega_4.$  For the second factor in $\Omega_4$, we use again that $A^{-1}$ is smoothing by two derivatives.  These considerations yield the bound
\begin{equation*}
    \Omega_4 
\leq c\norm{\partial^{s+1}_x  u^{\epsilon}}_{L^2} \norm{\partial^{s+1}_x u^{\epsilon}}_{L^2}
\leq cE.
\end{equation*}

We have now established $\sum_{k=1}^{4} \Omega_k \leq c(E+E^{\frac{3}{2}})$. Thus, we arrive at the corresponding 
bound for $\frac{dE_2}{dt},$
\begin{equation}\nonumber
\frac{dE_2}{dt} \leq c\left(E+E^{\frac{3}{2}}\right),
\end{equation}
and also for $\frac{dE}{dt},$
\begin{equation}\label{finalEnergyBound}
    \frac{dE}{dt} = \frac{dE_0}{dt}+\frac{dE_1}{dt}+\frac{dE_2}{dt} \leq c\left(E+E^{\frac{3}{2}}\right).
\end{equation}

We let $\overline{d} > 0$ be such that $E(0) \leq \overline{d}$. 
We ask on what interval of values of $t$ we may guarantee that $E(t) \leq 2\overline{d};$ for such values of $t,$ we have 
   \begin{equation*}
       \frac{dE}{dt} 
       \leq c\left(E+E^{\frac{3}{2}}\right)
       \leq c \left(2\overline{d}+(2\overline{d})^{\frac{3}{2}}\right).
   \end{equation*}
This implies that on an interval on which $E\leq 2\overline{d},$
  \begin{equation*}
      E 
      \leq c\left(2\overline{d}+(2\overline{d})^{\frac{3}{2}}\right)t + \overline{d}.
  \end{equation*}    
Thus, we can conclude that $E(t)\leq 2\overline{d}$ for all $t$ satisfying
 \begin{equation*}
      t\in\left[0,  \frac{\overline{d}}{c\left(2\overline{d}+(2\overline{d})^{\frac{3}{2}}\right)}\right].
 \end{equation*}
As this time interval is independent of $\epsilon,$ this completes the proof.
\end{proof}

\begin{remark}
We used several times above that the operator $A^{-1}$ is smoothing by two derivatives.  To be more precise,
since we are in the spatially periodic case we may use the Fourier series to see 
that $A^{-1}$ is a bounded linear operator between any space $H^{\ell}$ and 
$H^{\ell+2}.$  This is immediate because the $A$ operator here has constant coefficients.
In Section \ref{ACKSection} 
below, we will need to use an analogous operator, but in the more general case of non-constant coefficients.
This will be more involved, and understanding this inverse on certain function spaces (exponentially weighted Wiener algebras)
will be a significant focus of Section \ref{ACKSection}.
\end{remark}

\subsection{Well-posedness of the initial value problem}
In this section we establish the three elements of well-posedness (existence, uniqueness, and continuous dependence upon
the initial data) for the initial value problem for the non-mollified system \eqref{etaEquation}, \eqref{uEquation}.
We begin with existence, and will at the same time establish regularity of the solution.
In demonstrating the highest regularity (that the solution is continuous in time with values in $H^{s}\times H^{s+1}$),
we rely on the following elementary interpolation inequality; the proof of this may be found many places,
one of which is \cite{ambroseThesis}.
\begin{lemma}(Interpolation Inequality)\label{InterpolationInequality}
   Let $s' \geq 0$ and $s\geq s'$ be given. There exists $c>0$ such that for every $f\in H^{s},$ the following inequality holds:
   \begin{equation}\nonumber
   \norm{f}_{H^{s'}} \leq c\norm{f}^{1-\frac{s'}{s}}_{H^0} \norm{f}^{\frac{s'}{s}}_{H^s}.
   \end{equation}
\end{lemma}

The following is our existence theorem.  
\begin{thm}\label{existenceTheorem}
Let $s\in\mathbb{N}$ such that $s\geq2$ be given.  Let $\eta_{0}\in H^{s}$ and $u_{0}\in H^{s+1}$ be given.  Let $T>0$ be 
as in Lemma \ref{uniformTime}.  Then there exists $(\eta,u)\in C([0,T];H^{s}\times H^{s+1})$ which solves the initial value problem
\eqref{etaEquation}, \eqref{uEquation} with data $\eta(\cdot,0)=\eta_{0},$ $u(\cdot,0)=u_{0}.$
\end{thm}

\begin{proof}
The energy estimate we have established shows that $(\eta^\epsilon, u^\epsilon)$ is uniformly bounded in 
$C([0,T]: H^s\times H^{s+1}),$ with this $T$ being independent of $\epsilon.$ 
This implies that $(\eta^{\epsilon}_t, u^\epsilon_t)$ is uniformly bounded with respect to $\epsilon$ in 
$L^{\infty}\times L^{\infty}$ as well, when $s\geq2.$  This implies that the sequence $(\eta^\epsilon, u^\epsilon)$ is 
an equicontinuous family, and thus by the Arzela-Ascoli theorem there exists a subsequence (which we do not relabel)
$(\eta^\epsilon, u^\epsilon)$ which converges uniformly to some $(\eta, u)\in (C([0,2\pi]\times[0,T]))^{2}.$ 
We now establish regularity of this $(\eta,u)$ and
that $(\eta,u)$ is a solution of the non-regularized initial value problem.

Since the sequence $(\eta^{\epsilon},u^{\epsilon})$ is uniformly bounded with respect to both $\epsilon$ and $t$ in 
$H^{s}\times H^{s+1},$ and since the unit ball of a Hilbert space is weakly compact,  for any $t\in[0,T]$ we may find
a weak limit in $H^{s}\times H^{s+1}.$  Clearly this limit must again equal $(\eta,u),$ and thus we conclude that for every
$t,$ $(\eta(\cdot,t),u(\cdot,t))\in H^{s}\times H^{s+1},$ and that $(\eta,u)\in L^{\infty}([0,T];H^{s}\times H^{s+1}).$

Since $(\eta^{\epsilon},u^{\epsilon})$ converges to $(\eta,u)$ in $(C([0,2\pi]\times[0,T]))^{2},$ the convergence also holds in
$C([0,T];H^{0}\times H^{0}).$  Then using the uniform bound on $(\eta^{\epsilon},u^{\epsilon})$ 
provided by the proof of Lemma \ref{uniformTime}, and also using Lemma 
\ref{InterpolationInequality}, we see that the convergence also holds in $C([0,T];H^{s'}\times H^{s'+1}),$ for any $0\leq s'<s.$

We have concluded so far that the limit $(\eta,u)\in C([0,T];H^{s'}\times H^{s'+1})\cap L^{\infty}([0,T];H^{s}\times H^{s+1}).$
We can in fact show that $(\eta,u)\in C([0,T];H^{s}\times H^{s+1}),$ but we will delay this until after showing that $(\eta,u)$
solves the unregularized initial value problem.

To show that $(\eta,u)$ satisfies the appropriate system, we use the fundamental theorem of calculus on the approximate solutions,
\begin{align}\nonumber
   \eta^\epsilon(\cdot,t)
    =&\eta_0+\int_{0}^{t}
    -\frac{r_0}{2} u^{\epsilon}_x -\frac{1}{2}\eta^{\epsilon}u^{\epsilon}_x
    -\mathcal{J}_{\epsilon}\left((\mathcal{J}_{\epsilon}\eta^\epsilon_x) u^\epsilon\right) 
    \ d\tau,\\
    u^\epsilon(\cdot,t)\nonumber
    =&u_0+\int_{0}^{t} A^{-1} \left[
    -\overline{\beta}\eta^{\epsilon}_x-u^{\epsilon}u^{\epsilon}_x-\kappa u^{\epsilon}
    +\frac{\gamma\overline{\beta}r_0}{2}u^{\epsilon}_{xx}\right]\ d\tau.
\end{align}
We have established sufficient regularity to pass to the limit under the integrals, and thus we have
\begin{align}\nonumber
   \eta(\cdot,t)
    =&\eta_0+\int_{0}^{t} 
    \left[-\frac{r_0}{2} u_x -\frac{1}{2}\eta u_x-\eta_x u 
    \right]\ d\tau,\\
    u(\cdot,t)\nonumber
    =&u_0+\int_{0}^{t} A^{-1} 
    \left[-\overline{\beta}\eta_x-uu_x-\kappa u+\frac{\gamma\overline{\beta}r_0}{2}u_{xx}\right]\ d\tau.
\end{align}
Taking the derivative of these equations with respect to time, we see that $(\eta,u)$ does indeed satisfy the unregularized initial
value problem.

We now may demonstrate $(\eta,u)\in C([0,T];H^{s}\times H^{s+1}).$  By a standard argument (see, for example, the proof of
Theorem 3.4 of \cite{majdaBertozzi}) the uniform bound on solutions and the continuity in time in $H^{s'}\times H^{s}$ for 
all $0\leq s'<s$ implies weak continuity in time, i.e. $(\eta,u)\in C_{W}([0,T];H^{s}\times H^{s+1}).$  Since weak convergence
plus convergence of the norm implies convergence in a Hilbert space, all that remains to show, then, is continuity of the
$H^{s}\times H^{s+1}$ norm with respect to time.  To establish continuity of the norm, it is enough to establish right-continuity
at the initial time, $t=0.$ The general case (i.e. continuity of the norm at times other than the initial time) follows by considering
any other time to be a new initial time; by uniqueness of solutions, which is part of the content of Theorem \ref{CDOICTheorem} 
below, the solution starting from some time $t_{*}\in[0,T)$ is the same as the solution we have already found starting from $t=0.$
In this way, establishing right-continuity of the norm at the initial time demonstrates right-continuity of the norm at any time
in $[0,T).$  Left-continuity of the norm follows from time-reversibility of the equations.

So, as we have said, all that remains to establish is right-continuity of the $H^{s}\times H^{s+1}$ norm of the solution at 
$t=0.$  Weak continuity implies
\begin{equation}\label{liminfPart}
\liminf_{t\rightarrow 0^{+}}\|(\eta,u)\|_{H^{s}\times H^{s+1}}\geq \|(\eta_{0},u_{0})\|_{H^{s}\times H^{s+1}}.
\end{equation}
Similarly, for any $t\in[0,T],$ we have
\begin{equation}\nonumber
\limsup_{\epsilon\rightarrow0^{+}}\|(\eta^{\epsilon}(\cdot,t),u^{\epsilon}(\cdot,t))\|_{H^{s}\times H^{s+1}}\geq 
\|(\eta(\cdot,t),u(\cdot,t))\|_{H^{s}\times H^{s+1}}.
\end{equation}
Then, the energy estimate \eqref{finalEnergyBound} implies 
\begin{multline}\label{limsupPart}
\|(\eta_{0},u_{0})\|_{H^{s}\times H^{s+1}}\geq\limsup_{t\rightarrow 0^{+}}
\limsup_{\epsilon\rightarrow 0^{+}}\|(\eta^{\epsilon}(\cdot,t),u^{\epsilon}(\cdot,t))\|_{H^{s}\times H^{s+1}}
\\
\geq \limsup_{t\rightarrow 0^{+}}\|(\eta(\cdot,t),u(\cdot,t))\|_{H^{s}\times H^{s+1}}.
\end{multline}
Combining \eqref{liminfPart} and \eqref{limsupPart}, we have the conclusion.  This completes the proof of the theorem.
\end{proof}

Now, we will seek to establish the uniqueness of solutions $(\eta, u),$ and continuous dependence on the initial data.
\begin{thm}\label{CDOICTheorem}  
Let $(\eta_{0},u_{0})\in H^{s}\times H^{s+1}$ and $(\eta_{0}^{*},u_{0}^{*})\in H^{s}\times H^{s+1}$ be given.
Let $K>0$ be given such that 
\begin{equation}\nonumber
\|(\eta_{0},u_{0})\|_{H^{s}\times H^{s+1}}<K,\qquad \|(\eta_{0}^{*},u_{0}^{*})\|_{H^{s}\times H^{s+1}}<K.
\end{equation}
Let $T>0$ be such that there is a solution $(\eta,u)\in C([0,T];H^{s}\times H^{s+1})$ solving \eqref{etaEquation},
\eqref{uEquation} with initial data $\eta(\cdot,0)=\eta_{0},$ $u(\cdot,0)=u_{0}$ and such that there is a solution
$(\eta^{*},u^{*})\in C([0,T];H^{s}\times H^{s+1})$ solving \eqref{etaEquation}, \eqref{uEquation} with initial data
$\eta(\cdot,0)=\eta_{0}^{*},$ $u(\cdot,0)=u_{0}^{*},$ and such that
\begin{equation}\nonumber
\sup_{t\in[0,T]}\|(\eta(\cdot,t),u(\cdot,t))\|_{H^{s}\times H^{s+1}}
\leq K,\quad \sup_{t\in[0,T]}\|(\eta^{*}(\cdot,t),u^{*}(\cdot,t))\|_{H^{s}\times H^{s+1}}\leq K.
\end{equation}
 Then there exists $c>0,$ depending only on $K$ and $s,$ such that for any $s'\in[0,s),$
\begin{multline}\label{CDOICEstimate}
\sup_{t\in[0,T]}\left\|(\eta(\cdot,t)-\eta^{*}(\cdot,t),u(\cdot,t)-u^{*}(\cdot,t))\right\|_{H^{s'}\times H^{s'+1}}
\\
\leq c\left(\|(\eta_{0}-\eta_{0}^{*}, u_{0}-u_{0}^{*})\|_{H^{0}\times H^{1}}\right)^{1-s'/s}.
\end{multline}
In particular, solutions of the initial value problem for \eqref{etaEquation}, \eqref{uEquation} are unique.
\end{thm}

\begin{proof}
We define an energy for the difference of the two solutions, $Z,$ as
  \begin{equation}\nonumber
    Z(t) 
    = \int_{0}^{t} (\eta-\eta^*)^2 + (u-u^*)^2 + (u_x - u^*_x)^2\ dx.
 \end{equation}    
 We will estimate the growth of $Z.$
Its time derivative is 
 \begin{equation}\nonumber
     \frac{dZ(t)}{dt}= 2 \int_{0}^{t} (\eta-\eta^*)(\eta_t-\eta^*_t) + (u-u^*)(u_t-u^*_t) + (u_x-u^*_x)(u_x-u^*_x)_t \ dx. 
 \end{equation}
 We expand the time-derivatives as follows:
\begin{equation}\nonumber
    (\eta_t-\eta^*_t)
    =\  -\frac{r_0}{2}(u_x-u^*_x) - \frac{1}{2}(\eta u_x-\eta^*u^*_x) - (\eta_x u - \eta^*_x u^*)\nonumber
    = \sum_{k=1}^{3} z_k,
    \end{equation}
    \begin{multline}\nonumber
    (u_t -u^*_t)
    = \ A^{-1} \bigg[\overline{\beta}(\eta_x-\eta^*_x) +  (uu_x-u^*u^*_x)
     + \kappa(u-u^*) - \frac{\gamma\overline{\beta}r_0}{2}(u_{xx}-u^*_{xx})\bigg]\\
    = \sum_{k=4}^{7} z_k,
    \end{multline}
    \begin{equation}\nonumber
    (u_x-u^*_x)_t = \ (u_t-u^*_t)_x
    = \sum_{k=4}^{7} (z_k)_x.
\end{equation}
Thus, we have
\begin{equation}\label{dz}
    \frac{dZ(t)}{dt} = 2 \int_{0}^{t}\left[ (\eta -\eta^*)\sum_{k=1}^{3} z_k 
    +(u -u^*)\sum_{k=4}^{7} z_k + (u_x -u^*_x)\sum_{k=4}^{7} (z_k)_x\right]\ dx.
\end{equation}
 
 We begin by estimating the first term. Bounding each factor in $L^{2},$ we have 
 \begin{equation}\nonumber
     2 \int_{0}^{t} (\eta-\eta^*)(z_1)dx
    \leq c\norm{\eta-\eta^*}_{L^2}\norm{z_1}_{L^2}.
 \end{equation}
 Since $\norm{z_1}_{L^2}\leq c\norm{u-u^*}_{H^1}$
we also have
     $\norm{z_1}_{L^2} 
     \leq cZ(t)^{\frac{1}{2}}.$ 
This then implies
 \begin{equation}\nonumber
     2 \int_{0}^{t} (\eta-\eta^*)(z_1)dx
    \leq cZ(t).
 \end{equation}
 Next, we look at the term involving  $z_2$. We again bound each factor in $L^{2}:$
 \begin{equation}\nonumber
     2 \int_{0}^{t} (\eta-\eta^*)(z_2)dx
    \leq c \norm{\eta-\eta^*}_{L^2}\norm{z_2}_{L^2}.
 \end{equation}
Here we may bound $\norm{z_2}_{L^2}$ as follows:
 \begin{equation}\nonumber
     \norm{z_2}_{L^2}
     \leq  \frac{1}{2}\norm{(\eta u_x+(\eta^* u_x-\eta^* u_x)   -\eta^*u^*_x)}_{L^2}.
 \end{equation}
 Using the triangle inequality, this may be bounded as
 \begin{equation}\nonumber
     \norm{z_2}_{L^2} \leq c \norm{\eta u_x-\eta^* u_x}_{L^2} +c\norm{\eta^* u_x - \eta^* u^*_x}_{L^2}.
 \end{equation}
 We may bring out $u_x$ from the first term and $\eta^*$ from the second term:
 \begin{equation}\nonumber
     \norm{z_2}_{L^2} \leq c\norm{\eta-\eta^*}_{L^2}\norm{u_x}_{L^\infty} + c\norm{\eta^*}_{L^\infty}\norm{u - u^*}_{H^1}. 
 \end{equation}
By Sobolev embedding, since $s\geq1,$ we have
 $\norm{u_x}_{L^\infty} \leq c\norm{u}_{H^{s+1}}$ 
 and 
 $\norm{\eta^*}_{L^\infty} \leq c\norm{\eta}_{H^s}.$
We therefore have
 \begin{equation}\nonumber
     \norm{z_2}_{L^2} \leq c (\norm{u}_{H^{s+1}}+ \norm{\eta}_{H^s})Z(t)^{\frac{1}{2}}. 
     \end{equation}
Using the uniform bound on the solutions, we then have
\begin{equation}\nonumber
    2 \int_{0}^{t} (\eta-\eta^*)(z_2)dx \leq c Z(t).
\end{equation}
Next, we will look at the third term.
We begin by adding and subtracting in $z_{3},$
 \begin{equation}\nonumber
     z_3 = \eta^*_x u^* +(\eta^*_x u - \eta^*_x u) - \eta_x u.
 \end{equation}
We integrate by parts once, finding
\begin{equation}\nonumber
    2 \int_{0}^{t} (\eta-\eta^*)(z_3)dx
    = - 2 \int_{0}^{t} (\eta-\eta^*)(\eta^*)(u^*_x- u_x) dx  -2 \int_{0}^{t}\frac{1}{2} (\eta-\eta^*)^{2}(u_x)\ dx.
\end{equation}
This may then be bounded as
\begin{multline}\nonumber
    2 \int_{0}^{t} (\eta-\eta^*)(z_3)dx\\
    \leq c\norm{\eta-\eta^*}_{L^2}\norm{\eta^*}_{L^\infty}\norm{u^*- u}_{H^1} + c\norm{\eta-\eta^*}^{2}_{L^2}\norm{u_x}_{L^\infty}.
\end{multline}
Again using Sobolev embedding and the uniform bound, we have
\begin{equation}\nonumber
        2 \int_{0}^{t} (\eta-\eta^*)(z_3)dx
        \leq c Z(t).
\end{equation}

 To complete the proof, it is sufficient to prove 
\begin{equation}\nonumber
\norm{z_k}_{H^1} \leq cZ(t)^{\frac{1}{2}},\ \text{for}\  k = {4,5,6,7}.
\end{equation}
We begin with $z_{4}.$
Recall $A^{-1}$ is smoothing by two derivatives.  We may then say
 \begin{equation}\nonumber
     \norm{z_4}_{H^1} 
     \leq c\norm{\eta-\eta^*}_{H^0} 
     \leq cZ(t)^{\frac{1}{2}}.
 \end{equation}
To estimate $\norm{z_5}_{H^1},$ we begin by adding and subtracting,
 \begin{equation}\nonumber
    \norm{z_5}_{H^1} = \norm{A^{-1}(uu_x+(u^* u_x - u^* u_x)     -u^*u^*_x)}_{H^1}.
 \end{equation}
We use the triangle inequality and $\norm{\cdot}_{H^1} \leq \norm{\cdot}_{H^2}$, so that  we have
\begin{equation}\nonumber
    \norm{z_5}_{H^1} \leq \norm{A^{-1}(uu_x - u^* u_x)}_{H^2} + \norm{A^{-1}(u^* u_x -u^*u^*_x)}_{H^2}.
\end{equation}
Using the smoothing effect of $A^{-1},$ this becomes
\begin{equation}\nonumber
    \norm{z_5}_{H^1} 
    \leq c\norm{u - u^*}_{H^0}\norm{u_x}_{L^\infty} + c\norm{u^*}_{L^\infty}\norm{u_x -u^*_x}_{H^0}.
\end{equation}
Using the definition of $Z(t),$ Sobolev embedding, and the uniform bound, we conclude
\begin{equation}\nonumber
   \norm{z_5}_{H^1} 
   \leq c Z(t)^{\frac{1}{2}}.
\end{equation}
Proceeding similarly, we also have
 \begin{eqnarray*}
     \norm{z_6}_{H^1} 
     \leq cZ(t)^{\frac{1}{2}} \quad \text{and} \quad
     \norm{z_7}_{H^1} 
     \leq cZ(t)^{\frac{1}{2}}.
 \end{eqnarray*}
 
Hence, we may conclude that (\ref{dz}) is bounded as follows:
\begin{equation}\nonumber
    \frac{dZ(t)}{dt} 
    \leq cZ(t).
\end{equation}
By Gronwall's inequality, we therefore have that $Z(t)\leq e^{ct}Z(0)\leq e^{cT}Z(0).$  Together with 
Lemma \ref{InterpolationInequality},
this implies \eqref{CDOICEstimate}.
Uniqueness follows by taking $Z(0)=0,$ as this implies $Z(t)=0$ for $t>0.$  
\end{proof}

\section{Existence by the Cauchy-Kowalevski Theorem}\label{ACKSection}

In this section, we prove an existence theorem for the system \eqref{originalEtaEquation}, \eqref{originalUEquation},
making use of an abstract Cauchy-Kowalevski theorem.  This uses function spaces of analytic functions based on the Wiener
algebra.  We take this approach to proving existence because it is not clear that the system is well-posed in spaces of finite
regularity, i.e. Sobolev spaces. 

Consider the following simplified system, which is based upon the system \eqref{originalEtaEquation}, 
\eqref{originalUEquation} with $r_{0x}\neq0,$ 
keeping only the terms with the most derivatives, and simplifying to the constant coefficient case:
\begin{equation}\label{illPosedSystemExample}
\eta_{t}=u_{x},\qquad u_{t}+u_{xxt}=\eta_{xx}.
\end{equation}
This system is ill-posed, as can be demonstrated with a calculation in Fourier space.  Specifically, for any $k\in\mathbb{N},$
we have the solution 
\begin{equation}\label{etaIllPosed}
\eta(x,t)=\exp\left\{ikx+t\left(\frac{1-i}{\sqrt{2}}\right)\frac{k^{3/2}}{\sqrt{1+k^{2}}}\right\}+c.c,
\end{equation}
where as usual ``c.c.'' denotes the complex conjugate of the preceding.
The other component of the solution to this system, $u,$ may be inferred from this formula for $\eta$ and from the relation
$u_{x}=\eta_{t}.$  The ill-posedness of the initial value problem for the system \eqref{illPosedSystemExample} may be seen
directly from \eqref{etaIllPosed}, as the exponential growth rate in time grows without bound as $k$ goes to infinity.

The above heuristic argument suggests ill-posedness in Sobolev spaces.  Furthermore, while the derivation of the model
in \cite{mitsotakis1}
is based  an asymptotic expansion of the velocity potential, Boussinesq models are frequently derived by instead making long-wave 
approximations;  in such a long-wave model, one
typically expects coefficients like $r_{0}(x)$ to be homogenized in the long-wave limit, see for instance \cite{dougShariCoops} 
for an example of this phenomenon in the case of long-wave limits of polyatomic lattices.  In this sense, while the system
\eqref{originalEtaEquation}, \eqref{originalUEquation} is more general, the system \eqref{etaEquation}, \eqref{uEquation}
may be more fundamental.

\subsection{The abstract Cauchy-Kowalevski theorem of Kano and Nishida}

The following abstract Cauchy-Kowalevski theorem is proved by Kano and Nishida \cite{kanoNishida}; other, related abstract
Cauchy-Kowalevski theorems may be found in \cite{caflisch}, \cite{nirenberg}, \cite{nishida}.

\begin{thm}[Cauchy-Kowalevski Theorem] \label{CKtheorem} Let $\{B_\rho\}_{\rho\geq0}$ be a scale of Banach spaces, 
such that for any $\rho$,  $B_\rho$ is a linear subspace of $B_0$. Suppose that 
\begin{equation}\label{rhoConditions}
    B_{\rho} \subset B_{\rho'}, \quad \norm{\cdot}_{\rho'} \leq \norm{\cdot}_{\rho}\quad \quad \text{for}\ \rho' \leq \rho,
\end{equation}
where $\norm{\cdot}_{\sigma}$ denotes the norm of $B_{\sigma}$ for any $\sigma\geq0.$ We assume the following conditions:\\

\noindent
\textbf{(H1)} There exist constants $R>R_0>0$, $T>0,$ and $\rho_{0}>0,$ 
such that for any $0\leq \rho' <\rho <\rho_0$, 
$(z,t,s) \mapsto F(z,t,s)$
is a continuous operator of 
\begin{equation}\nonumber
    \{z \in B_{\rho}:\norm{z}_{\rho}<R\} \times \{0\leq t\leq T\}\times \{0\leq s\leq T \} \quad \text{into} \ B_{\rho'}.
\end{equation} 
\textbf{(H2)} For every $\rho<\rho_0$,  $F(0,t,s)$ is a continuous function of t with values in $B_\rho$ and satisfies with a fixed constant $K$,
\begin{equation}\nonumber
    \norm{F(0,t,s)}_{\rho} 
    \leq \frac{K}{\rho_0-\rho}.
\end{equation}
\textbf{(H3)} For any $0\leq\rho'<\rho<\rho_0$ and all $z,\tilde{z} \in B_{\rho}$, with $\norm{z}_{\rho}<R$, $\norm{\tilde{z}}_{\rho}<R$, F satisfies the following for all $t, s$ in $[0,T]$,
\begin{equation}\nonumber
    \norm{F(z,t,s)-F(\tilde{z},t,s)}_{\rho'} \leq \frac{C \norm{z-\tilde{z}}_{\rho}}{\rho-\rho'}
\end{equation}
with a constant $C$ independent of \{$t,s,z,\tilde{z},\rho$,$\rho'$\}.\\
If \textbf{(H1)-(H3)} hold, there exists a positive constant $\lambda$ such that we have the unique continuous solution of
\begin{equation}\label{kanoNishidaProblem}
z(t)=z_{0}(t)+\int_{0}^{t}F(z(s),t,s)\ ds,
\end{equation}
 for all $0 < \rho <\rho_0$ and $|t|<\lambda(\rho_0-\rho),$
 with value in $B_{\rho}.$ 
\end{thm}

\begin{remark} We have stated the conclusion as in \cite{kanoNishida}, but it can be rephrased in a way we will find more helpful.
The existence of $\lambda$ and the conditions $0<\rho<\rho_{0}$ and $|t|<\lambda(\rho_{0}-\rho)$ are equivalent to the existence
of an upper bound on the time of existence for solutions.  Namely, the solution can be continued to a time $t>0$ 
as long as there exists a value
$\rho\in(0,\rho_{0})$ such that $t<\lambda(\rho_{0}-\rho).$  Thus the solution exists on the interval $[0,T),$ where 
$T=\lambda\rho_{0}.$  In Theorem \ref{secondMain} below, rather than concluding the existence of $\lambda,$ we will conclude
the existence of $T.$
\end{remark}

\begin{remark}
The form of the equation \eqref{kanoNishidaProblem} that Kano and Nishida consider is clearly intended to allow for 
semigroups from linear operators, such as would appear in a parabolic problem.  In our intended application, we have
no such semigroup present, so we do not need both variables $s$ and $t.$  In particular, adapting the initial value
problem for the system \eqref{originalEtaEquation}, \eqref{originalUEquation} to the form \eqref{kanoNishidaProblem},
we get $F=(F_{1},F_{2}),$ where $F_{1}$ and $F_{2}$ are given by (suppressing dependence on the spatial variable)
\begin{equation}\label{f1}
F_{1}((\eta(s),u(s)),t,s)=-\frac{1}{2}(r_{0}+\eta(s))u_{x}(s)-(r_{0}+\eta(s))_{x}u(s),
\end{equation}
\begin{multline}\label{f2}
F_{2}((\eta(s),u(s)),t,s)=A^{-1}\Bigg[
-(\bar{\beta}\eta(s))_{x}-u(s)u_{x}(s)\\
-\frac{(3\bar{\alpha}+r_{0})r_{0x}}{2}(\bar{\beta}\eta(s))_{xx}-\kappa u(s)
+\gamma\bar{\beta}\left(r_{0x}u(s)+\frac{r_{0}}{2}u_{x}(s)\right)
\Bigg].
\end{multline}
Here, the operator $A$ is given by
\begin{equation}\nonumber
Af=\left[1-\bar{\alpha}r_{0xx}-\frac{(4\bar{\alpha}+r_{0})r_{0}}{8}\partial_{x}^{2}\right]f.
\end{equation}
\end{remark}

We will use the exponentially weighted Wiener algebras as our spaces $B_{\rho};$ given $\rho\geq0,$ we say $f\in B_{\rho}$
if and only if 
\begin{equation}\nonumber
\|f\|_{\rho}=\sum_{k\in\mathbb{Z}}e^{\rho|k|}|\hat{f}_{k}|<\infty,
\end{equation}
where $\{\hat{f}_{k}\}$ are the Fourier coefficients of $f.$  These spaces satisfy \eqref{rhoConditions}.  Furthermore, these
spaces are Banach algebras, so that if $f\in B_{\rho}$ and $g\in B_{\rho},$ then $fg\in B_{\rho},$ with
\begin{equation}\label{algebraEstimate}
\|fg\|_{\rho}\leq \|f\|_{\rho}\|g\|_{\rho}.
\end{equation}
These spaces also have the Cauchy estimate; if $0\leq\rho'<\rho,$ then for all $f\in B_{\rho},$
\begin{equation}\label{cauchyEstimate}
\|\partial_{x}f\|_{\rho'}\leq\frac{e}{\rho-\rho'}\|f\|_{\rho}.
\end{equation}

Now, we move on to show that abstract Cauchy-Kowalevski theorem applies to our system \eqref{originalEtaEquation},
\eqref{originalUEquation}, by showing that $F_{1},$ $F_{2}$ given in \eqref{f1}, \eqref{f2} satisfy the hypothesis {\bf{(H1)-(H3)}}. 
Before verifying the hypotheses {\bf(H1)-(H3)}, 
it is important to understand the action of the operator $A^{-1}$ on the scale of spaces
$B_{\rho};$ we consider this inverse operator next in Section \ref{inverseSection}, and then verify {\bf(H1)-(H3)} in Section 
\ref{verificationSection}.
In bounding the inverse operator,
we necessarily make some assumptions on the function $r_{0}.$  We will show that the assumptions on $r_{0}$ 
may be satisfied by furnishing a family of examples in Section \ref{exampleSection} below.  

\subsection{The inverse operator}\label{inverseSection}

In the analysis of Section \ref{wellPosednessSection} 
above, we used the fact that $(1-\partial_{x}^{2})^{-1}$ is a bounded linear operator from 
$H^{s}$ to $H^{s+2};$ it of course is also a bounded linear operator from $B_{\rho}$ to itself, for any $\rho.$  
The inverse operator we must
deal with, however, is more complicated than $(1-\partial_{x}^{2})^{-1},$ as we now have to account for non-constant coefficients.

We begin by rewriting $A$ to factor out the function multiplying $\partial_{x}^{2},$
\begin{equation}\nonumber
A=g_{1}\left[g_{2}-\partial_{x}^{2}\right],
\end{equation}
where the functions $g_{1}$ and $g_{2}$ are given by
\begin{equation}\nonumber
g_{1}=\frac{(4\bar{\alpha}+r_{0})r_{0}}{8},\qquad g_{2}=\frac{8(1-\bar{\alpha}r_{0xx})}{(4\bar{\alpha}+r_{0})r_{0}}.
\end{equation}
We make the following assumptions on $g_{1}$ and $g_{2}$ (of course, these are really assumptions about the function $r_{0}$).\\

\noindent{\bf (H4)}  We assume $g_{1}>0$ and there
exists $\rho_{0}>0$ and $c_{0}>0$ such that $\partial_{x}^{j}g_{1}^{-1}\in B_{\rho_{0}}$ for $j\in\{0,1,2\},$
$g_{2}\in B_{\rho_{0}},$ and 
\begin{equation}\label{inverseNormAssumption}
\left\|\frac{g_{2}-c_{0}}{c_{0}}\right\|_{\rho_{0}}<1.
\end{equation}

We will focus now on inverting $g_{2}-\partial_{x}^{2}.$  
We make the decomposition $g_{2}-\partial_{x}^{2}=A_{1}+A_{2},$ where
\begin{equation}\nonumber
A_{1}=g_{2}-c_{0},\qquad A_{2}=c_{0}-\partial_{x}^{2}.
\end{equation}
As we are interested in $(A_{1}+A_{2})^{-1},$ we invert the identity $A_{1}+A_{2}=(1+A_{1}A_{2}^{-1})A_{2}$ to find the formula
\begin{equation}\nonumber
(A_{1}+A_{2})^{-1}=A_{2}^{-1}(1+A_{1}A_{2}^{-1})^{-1}.
\end{equation}

It is easily verified that, for any $0\leq\rho\leq\rho_{0},$ $A_{2}^{-1}$ is bounded from $B_{\rho}$ to itself, 
with operator norm $1/c_{0}.$
This implies that $A_{1}A_{2}^{-1}$ is also bounded from $B_{\rho}$ to itself, with 
\begin{equation}\nonumber
\|A_{1}A_{2}^{-1}\|_{B_{\rho}\rightarrow B_{\rho}}
\leq
\left\|\frac{g_{2}-c_{0}}{c_{0}}\right\|_{\rho}
\leq
\left\|\frac{g_{2}-c_{0}}{c_{0}}\right\|_{\rho_{0}}.
\end{equation}
Thus by \eqref{inverseNormAssumption}, the operator norm of $A_{1}A_{2}^{-1}$ is strictly less than $1.$  The operator
$1+A_{1}A_{2}^{-1}$ can therefore be inverted by Neumann series.  We conclude that $A^{-1}$ is well-defined as a bounded linear 
operator mapping $B_{\rho}$ to $B_{\rho}.$  We have proved the following lemma.

\begin{lemma}\label{inverseLemma}
Assume {\bf(H4)}.  For all $\rho$ satisfying $0\leq\rho\leq\rho_{0},$
 $A^{-1}$ is a well-defined bounded linear operator mapping from $B_{\rho}$ to
$B_{\rho}.$
\end{lemma}

We also have the following corollary.
\begin{corollary}\label{inverseCorollary}
Assume {\bf(H4)}.  For any $0\leq\rho'<\rho\leq\rho_{0},$ the operators $A^{-1}\partial_{x}$ and $A^{-1}\partial_{x}^{2}$ are 
bounded from $B_{\rho}$ to
$B_{\rho'},$ with the estimates
\begin{equation}\nonumber
\|A^{-1}\partial_{x}^{j}f\|_{\rho'}\leq \frac{c\|f\|_{\rho}}{\rho-\rho'},\quad j\in\{1,2\}.
\end{equation}
\end{corollary}
\begin{proof} We focus on the case $j=2,$ as the other case is simpler.  
We write $A^{-1}=(g_{2}-\partial_{x}^{2})^{-1}g_{1}^{-1},$ and for $f\in B_{\rho_{0}},$ we have
\begin{multline}\nonumber
A^{-1}\partial_{x}^{2}f=(g_{2}-\partial_{x}^{2})^{-1}g_{1}^{-1}\partial_{x}^{2}f
\\
=(g_{2}-\partial_{x}^{2})^{-1}\left[(\partial_{x}^{2}(g_{1}^{-1}f)-2(\partial_{x}g_{1}^{-1})(\partial_{x}f)-(\partial_{x}^{2}g_{1}^{-1})f\right].
\end{multline}
We then add and subtract, finding
\begin{multline}\nonumber
A^{-1}\partial_{x}^{2}f=(g_{2}-\partial_{x}^{2})^{-1}\left[\left(\partial_{x}^{2}-g_{2}\right)\left(g_{1}^{-1}f\right)\right]
+(g_{2}-\partial_{x}^{2})^{-1}(g_{2}g_{1}^{-1}f)
\\
-2(g_{2}-\partial_{x}^{2})^{-1}\left((\partial_{x}g_{1}^{-1})(\partial_{x}f)\right)
-(g_{2}-\partial_{x}^{2})^{-1}\left((\partial_{x}^{2}g_{1}^{-1})f\right).
\end{multline}
The first term on the right-hand side simplifies, and this becomes
\begin{multline}\nonumber
A^{-1}\partial_{x}^{2}f=-g_{1}^{-1}f
+(g_{2}-\partial_{x}^{2})^{-1}(g_{2}g_{1}^{-1}f)
\\
-2(g_{2}-\partial_{x}^{2})^{-1}\left((\partial_{x}g_{1}^{-1})(\partial_{x}f)\right)
-(g_{2}-\partial_{x}^{2})^{-1}\left((\partial_{x}^{2}g_{1}^{-1})f\right).
\end{multline}
There are four terms on the right-hand side, three of which involve zero derivatives of $f$ and one of which involves $\partial_{x}f.$
All operators applied here either to $f$ or to $\partial_{x}f$ are bounded, and we may then use the inequalities
\eqref{rhoConditions} and \eqref{cauchyEstimate} to reach the conclusion.
\end{proof}

As we have said, we will discuss functions $r_{0}$ which satisfy {\bf(H4}) in Section \ref{exampleSection} below.

\subsection{Verifying the hypotheses}\label{verificationSection}

We are now in a position to state our existence theorem for the initial value problem for the system \eqref{originalEtaEquation},
\eqref{originalUEquation}.

\begin{thm}\label{secondMain}
Assume that $r_{0}$ satisfies {\bf(H4)}.  Furthermore assume $r_{0},$ $r_{0x},$ $\bar{\beta},$ and $\bar{\beta}_{x}$ are all
in $B_{\rho_{0}}.$  Let $\eta_{0}\in B_{\rho_{0}}$ and $u_{0}\in B_{\rho_{0}}$ be given.
Then there exists $T>0$ such that there exists a solution $(\eta,u)$ of the initial value problem \eqref{originalEtaEquation},
\eqref{originalUEquation} with initial conditions $\eta(\cdot,0)=\eta_{0},$ $u(\cdot,0)=u_{0},$ on the time interval $[0,T).$
At each time $t\in[0,T],$ each of $\eta(\cdot,t)$ and $u(\cdot,t)$ belong to the space $B_{\rho}$ for all 
$0\leq\rho<\rho_{0}\left(1-\frac{t}{T}\right).$
\end{thm}

\begin{proof}
With the estimates we have established, it is immediate that $F=(F_{1},F_{2})$ maps $B_{\rho}$ to $B_{\rho'},$ for 
$0<\rho'<\rho<\rho_{0}.$ This establishes {\bf(H1)}.  Next, clearly {\bf(H2)} is automatically satisfied, as $F((0,0),t,s)=0.$
What remains, then, is to establish {\bf(H3)}.

To begin to verify {\bf(H3)}, we consider $F_{1}((\eta,u),t,s)-F_{1}(\tilde{\eta},\tilde{u}),t,s),$
\begin{multline}\nonumber
      \norm{-\frac{r_0}{2}(u_x-\tilde{u}_x) - \frac{1}{2}(\eta u_x - \tilde{\eta} \tilde{u}_x) - (\eta_xu -\tilde{\eta}_x\tilde{u})}_{\rho'}
    \\
     \leq \norm{-\frac{r_0}{2}(u_x-\tilde{u}_x)}_{\rho'} +\norm{\frac{1}{2}(\eta u_x - \tilde{\eta} \tilde{u}_x)}_{\rho'}
     +\norm{-r_{0x}(u-\tilde{u})}_{\rho'}
     + \norm{\eta_xu -\tilde{\eta}_x\tilde{u}}_{\rho'}
\\     \leq I + II + III+IV.
\end{multline}
The first term, $I,$ is readily bounded using the Cauchy estimate \eqref{cauchyEstimate},
\begin{equation}\nonumber
    I \leq  \frac{c\norm{u-\tilde{u}}_{\rho}}{\rho-\rho'}.
\end{equation}
For the second term, $II,$ we add and subtract, and use the algebra property \eqref{algebraEstimate},
\begin{multline}\nonumber
    II \leq {\frac{1}{2}}\norm{\eta u_x - \tilde{\eta}u_x + \tilde{\eta}u_x -\tilde{\eta}\tilde{u}_x}_{\rho'}
    \\
    \leq {\frac{1}{2}} \norm{u_x}_{\rho'}\norm{\eta-\tilde{\eta}}_{\rho'} + {\frac{1}{2}}\norm{\tilde{\eta}}_{\rho'}\norm{u_x -\tilde{u}_x}_{\rho'}. 
\end{multline}
We may then apply the Cauchy estimate \eqref{cauchyEstimate}, finding
\begin{equation}\nonumber
    II \leq c \cdot \frac{\norm{\eta - \tilde{\eta}}_{\rho}+\norm{u-\tilde{u}}_{\rho}}{\rho-\rho'}. 
\end{equation}
The third and fourth terms, $III+IV,$ may be estimated similarly, using the assumption on $r_{0x}$ for the estimate for $III,$
\begin{equation}\nonumber
    III+IV
    \leq c\cdot \frac{\norm{\eta-\tilde{\eta}}_{\rho}+\norm{u-\tilde{u}}_{\rho}}{\rho-\rho'}.
\end{equation}

We now consider $F_{2}.$  Specifically, we estimate $F_{2}((\eta,u),t,s)-F_{2}((\tilde{\eta},\tilde{u}),t,s):$
\begin{equation}\nonumber
\|F_{2}((\eta,u),t,s)-F_{2}((\tilde{\eta},\tilde{u}),t,s)\|_{\rho'}
\leq V + VI + VII+VIII+IX+X,
\end{equation}
where
\begin{equation}\nonumber
V=\left\|A^{-1}\partial_{x}(\bar{\beta}(\eta-\tilde{\eta}))\right\|_{\rho'},
\end{equation}
\begin{equation}\nonumber
VI=\|A^{-1}(uu_{x}-\tilde{u}\tilde{u}_{x})\|_{\rho'},
\end{equation}
\begin{equation}\nonumber
VII=\left\|A^{-1}\left(\frac{(3\bar{\alpha}+r_{0})r_{0x}}{2}\partial_{x}^{2}(\bar{\beta}(\eta-\tilde{\eta}))\right)\right\|_{\rho'},
\end{equation}
\begin{equation}\nonumber
VIII=\kappa\|A^{-1}(u-\tilde{u})\|_{\rho'},
\end{equation}
\begin{equation}\nonumber
IX=\gamma\|A^{-1}(\bar{\beta}r_{0x}u)\|_{\rho'},
\end{equation}
\begin{equation}\nonumber
X=\frac{\gamma}{2}\|A^{-1}(\bar{\beta}r_{0}\partial_{x}(u-\tilde{u}))\|_{\rho'}.
\end{equation}
We will omit some details, but each of $V,$ $VI,$ $VII,$ $VIII,$ $IX,$ and $X$ is bounded appropriately. 
We will demonstrate the estimate for a few terms, specifically for $V,$ $VI,$ and $X.$  The remaining terms are similar.

By Corollary \ref{inverseCorollary} and \eqref{algebraEstimate}, and by assumption on $\bar{\beta},$ we have
\begin{equation}\nonumber
    V \leq \frac{c\|\bar{\beta}\|_{\rho}\|\eta_{x}-\tilde{\eta}_{x}\|_{\rho}}{\rho-\rho'}
    \leq \frac{c\|\eta_{x}-\tilde{\eta}_{x}\|_{\rho}}{\rho-\rho'}.
\end{equation}
For $VI,$ we notice $uu_x = \frac{1}{2}\partial_x(u^2)$. Thus, we may write
\begin{equation}\nonumber
    VI = \frac{1}{2}\norm{A^{-1}\partial_{x}(u^2-\tilde{u}^2)}_{\rho'}.
\end{equation}
We then use Corollary \ref{inverseCorollary}, finding
\begin{equation}\nonumber
    VI
    \leq \frac{c\norm{u^{2}-\tilde{u}^{2}}_{\rho}}{\rho-\rho'}.
\end{equation}
Application of the algebra property \eqref{algebraEstimate} then yields
\begin{equation}\nonumber
VI\leq \frac{c\|u-\tilde{u}\|_{\rho}}{\rho-\rho'}.
\end{equation}
The final term for which we will provide details is $X.$  We write 
\begin{equation}\nonumber
\bar{\beta}r_{0}(u_{x}-\tilde{u}_{x})=\partial_{x}\left(\bar{\beta}r_{0}(u-\tilde{u})\right)-(\bar{\beta}r_{0})_{x}(u-\tilde{u}),
\end{equation} 
and we bound $X$ as
\begin{equation}\nonumber
X\leq c\|A^{-1}\partial_{x}(\bar{\beta}r_{0}(u-\tilde{u}))\|_{\rho'}
+c\|A^{-1}((\bar{\beta}r_{0})_{x}(u-\tilde{u}))\|_{\rho'}=X_{1}+X_{2}.
\end{equation}
We use Corollary \ref{inverseCorollary}, the algebra property \eqref{algebraEstimate}, and our assumptions on 
$r_{0}$ and $\bar{\beta}$ to bound $X_{1},$ finding
\begin{equation}\nonumber
X_{1}\leq \frac{c\left\|\bar{\beta}r_{0}(u-\tilde{u})\right\|_{\rho}}{\rho-\rho'}\leq\frac{c\|u-\tilde{u}\|_{\rho}}{\rho-\rho'}.
\end{equation}
For $X_{2},$ we use Lemma \ref{inverseLemma}, the algebra property \eqref{algebraEstimate}, and our assumptions on
$r_{0}$ and $\bar{\beta}$ to find
\begin{equation}\nonumber
X_{2}\leq c\|(\bar{\beta}r_{0})_{x}(u-\tilde{u})\|_{\rho'}
\leq 
c\|u-\tilde{u}\|_{\rho}
\leq
\frac{c\|u-\tilde{u}\|_{\rho}}{\rho-\rho'}.
\end{equation}
The remaining terms are similar.
This concludes the proof.
\end{proof}

\subsection{A family of examples}\label{exampleSection}

Of course we wish to show that the set of functions which satisfy {\bf(H4)} is nonempty; it is trivially nonempty since
constant functions $r_{0}$ satisfy it.  Going further, we wish to show that there are also non-constant functions which 
satisfy {\bf(H4)}.  To this end 
we now demonstrate a simple family of functions $r_{0}$ which satisfy {\bf(H4)}. 

 Let $R_{0}>0;$ we consider
$r_{0}=R_{0}+\varepsilon\sin(x),$ for sufficiently small $\varepsilon.$  
First, clearly, for sufficiently small $\varepsilon,$ we have $g_{1}>0,$ as required.  Second, we see that 
for any $\rho\geq0,$ the function $8(1-\bar{\alpha})r_{0xx}$ is in $B_{\rho}.$

We next demonstrate that there exist values of $\rho>0$ such that $\frac{1}{r_{0}}\in B_{\rho}.$  We denote $\psi=\frac{1}{r_{0}},$
and let the Fourier coefficients of $\psi$ be denoted as $\hat{\psi}_{k}.$  We adapt this argument from the proof of Theorem IX.13 in
\cite{reedSimon}.

Clearly, $\psi$ has analytic extension to a strip of width $N>0$ in the complex plane, 
for some $N>0$ (we can even explicitly calculate
this $N$ if so desired); we call this extension $\tilde{\psi}.$  
For any given $\rho$ such that $0<\rho<N,$ we denote by $\psi_{\rho}$ the function such that $\psi_{\rho}(x)=\tilde{\psi}(x+i\rho),$
and we denote its Fourier coefficients as $\hat{\psi}_{\rho k}.$  Since $\psi_{\rho}$ is a bounded function on the torus, of course
there exists $C>0$ such $|\hat{\psi}_{\rho k}|\leq C.$  By the Cauchy Integral Theorem, 
we have $\hat{\psi}_{k}=e^{-\rho k}\hat{\psi}_{\rho k}.$
Thus, for $k\geq 0,$ we have $|\hat{\psi}_{k}|\leq Ce^{-\rho |k|}.$  Negative values of $k$ can be treated similarly.
This implies that for any $0\leq\rho'<\rho,$ we have $\psi\in B_{\rho'}.$  Since $\rho$ may be taken arbitrarily close to $N,$ we
conclude that for any $\rho\in(0,N),$ we have $\psi\in B_{\rho}.$

A similar argument naturally applies to the function $\frac{1}{(4\bar{\alpha}+r_{0})},$ and to derivatives of $\frac{1}{r_{0}}.$ 
Finally, by the algebra property for the $B_{\rho}$ spaces, 
we conclude that there exists $\rho_{0}>0$ such that $\partial_{x}^{j}g_{1}^{-1}$ and $g_{2}$ are all in $B_{\rho_{0}}.$   

Next we consider existence of the constant $c_{0}$ such that \eqref{inverseNormAssumption} holds.  Denoting 
$K_{0}=\frac{8}{(4\bar{\alpha}+R_{0})R_{0}},$ we see that as $\varepsilon\rightarrow0,$ for any $c_{0}\in(0,K_{0}),$ we have
\begin{equation}\nonumber
\left\|\frac{g_{2}-c_{0}}{c_{0}}\right\|_{\rho_{0}}\rightarrow \frac{K_{0}-c_{0}}{c_{0}}.
\end{equation}
As long as $c_{0}\in (0,\frac{K_{0}}{2}),$ for sufficiently small $\varepsilon,$ we see that \eqref{inverseNormAssumption} holds.

We have therefore demonstrated that the set of functions $r_{0}$ satisfying {\bf(H4)} is nontrivial.  In Theorem \ref{secondMain},
we also made the further assumption that $r_{0},$ $r_{0x},$ $\bar{\beta},$ and $\bar{\beta}_{x}$ are all in $B_{\rho_{0}}.$  Clearly
these properties hold as well (recall that $\bar{\beta}$ is proportional to $\psi^{2}$) for our family
of examples.

\section{Existence of periodic traveling waves}\label{travelingSection}

In this section, we establish the existence of periodic traveling waves for the system \eqref{etaEquation}, \eqref{uEquation}.
We will do this in the case $\kappa=\gamma=0.$  We prove existence by means of the following local bifurcation theorem
\cite{zeidler}:

\begin{thm}[Bifurcation Theorem]\label{bifurcationTheorem}
Let $\mathcal{H}'$ and $\mathcal{H}$ be Hilbert spaces, and let $(\eta_0,u_0) \in \mathcal{H}'$. Let $U$ be an open neighborhood of $(\eta_0,u_0)$ in $\mathcal{H}'$. Suppose\\
\textbf{(B1)} The map $\phi$ : $U\times\mathbb{R} \rightarrow \mathcal{H}$ is $C^2$.\\
\textbf{(B2)} For all $c \in \mathbb{R}$, $\phi((\eta_0,u_0),c) = 0$.\\
\textbf{(B3)} For some $c_0$, $L(c_0):= \partial_{(\eta,u)}\phi((\eta_0,u_0),c)$ has a one-dimensional kernel and has zero Fredholm index.\\ 
\textbf{(B4)} If $h' \in \mathcal{H}'$ spans the kernel of $L(c_0)$ and $h^* \in \mathcal{H}$ spans the kernel of $L^*(c_0)$, then $\left<h^*, \partial_c L(c_0)h'\right>_{\mathcal{H}} \neq 0.$\\
If these four conditions hold, then there exists a sequence $\{(\eta_n, u_n),c_n\}_{n\in\mathbb{N}}\subset \mathcal{H}'\times\mathbb{R}$ with\\
a. $\lim_{n\rightarrow\infty} {((\eta_n,u_n),c_n)} = ((\eta_0,u_0),c_0)$.\\
b. $(\eta_n,u_n) \neq (\eta_0,u_0)$ for all $n \in \mathbb{N}$ and\\
c. $\phi((\eta_n,u_n,)c_n) = 0.$
\end{thm}

We will use Theorem \ref{bifurcationTheorem} to prove the following theorem:
\begin{thm}\label{travelingWaveExistenceTheorem}
There exists a non-zero sequence  $\{(\eta_{n}(x,t),u_{n}(x,t))\}_{n\in\mathbb{N}}$ such that for all $n,$ for all $t,$
$\eta_{n}\in H^{1}(\mathbb{T}),$ $u_{n}\in H^{3}(\mathbb{T}),$ 
and there exists a sequence of real numbers $c_{n}$ such that for all $n,$ 
the functions $(\eta_{n},u_{n})$ constitute 
a nontrivial traveling wave solution of \eqref{etaEquation}, \eqref{uEquation} with speed $c_{n}.$
There exists $c_{\infty}$ such that as $n\rightarrow\infty,$ $c_{n}\rightarrow c_{\infty}$ and $(\eta_{n},u_{n})\rightarrow(0,0).$
Furthermore, at each time, each of $\eta_{n}$ and $u_{n}$ are even with zero mean.
\end{thm}

The rest of this section is the proof of Theorem \ref{travelingWaveExistenceTheorem}.
Specifically, we now demonstrate that the conditions \textbf{(B1)}, \textbf{(B2)}, \textbf{(B3)}, and \textbf{(B4)} hold for the traveling 
wave equations for \eqref{etaEquation}, \eqref{uEquation}.
This will be the content of Sections \ref{b1B2Section}, \ref{b3Section}, and \ref{b4Section}.

\subsection{The mapping, \textbf{(B1)}, and \textbf{(B2)}}\label{b1B2Section}
Part of establishing that \textbf{(B1)} holds is specifying the function spaces and the mapping to be studied.
We begin by defining the space $\mathcal{H}'.$  We consider  symmetric solutions, so we let
\begin{equation*}
    \mathcal{H}' := H^1_{e,0}\times H^3_{e,0}
\end{equation*}
where for any $s,$
\begin{equation*}
    \mathcal{H}^s_{e,0} := \left\{f \in H^s : f \ \text{is even  and} \int_{0}^{M} f(x) dx = 0\right\}.
\end{equation*}
We recall that $H^s$ indicates the spatially periodic $L^2$-based Sobolev space of index $s.$
As our choice of spaces shows, 
we will look for solutions $\eta$ and $u$ where both are even functions and have zero mean value. 

We now give the traveling wave ansatz,
\begin{equation}\nonumber
 \eta  = \eta(x-ct), \qquad  u = u(x-ct),
\end{equation}
for some $c\in\mathbb{R}.$
With this ansatz, and with parameter values $\kappa=0$ and $\gamma=0,$ then the system \eqref{etaEquation}, 
\eqref{uEquation} becomes
\begin{eqnarray}
\label{travelingEta}       -c\eta' + \frac{1}{2}r_0u' + \frac{1}{2}\eta u' + \eta'u = 0, \\
\label{travelingU}       -cu' + \overline{\beta}\eta' + uu' + \frac{c(4\overline{\alpha}+r_0)r_0}{8}u''' = 0.
  \end{eqnarray}
The mapping $\phi(\eta,u)$ is given by the left-hand sides of  \eqref{travelingEta}, \eqref{travelingU}.

The space $\mathcal{H}'$ maps to odd functions under $\phi.$ Therefore, we take the codomain $\mathcal{H}$ to be
\begin{equation*}
    \mathcal{H} := L^2_{odd}\times L^2_{odd}.
\end{equation*}

With these definitions, {\bf(B1)} and {\bf(B2)} clearly hold, with the trivial solutions being $\eta=u=0$ for any $c\in\mathbb{R}.$

\subsection{The linearized operator and {\bf(B3)}}\label{b3Section}

We linearize the system about the equilibrium $\eta=0,$  $u=0.$ 
The linearization of the $\eta$ equation is 
\begin{equation}\label{etaLinearization}
\eta_{1}' = \frac{r_0}{2c}u_{1}',
\end{equation}
which we may integrate, using the fact that our function spaces specify zero mean, finding
\begin{equation}\nonumber
 \eta_{1} =\frac{r_0}{2c}u_{1}.
\end{equation}
We turn to the $u$ equation, which linearizes as
\begin{equation}\label{uLinearization}
\frac{c(4\overline{\alpha}+r_0)r_0}{8}u_{1}''' -cu_{1}'+ \overline{\beta}\eta_{1}'  = 0.
\end{equation}
The system \eqref{etaLinearization}, \eqref{uLinearization} is our linearized system, and the left-hand sides define our
linearized operator, $L.$  To investigate the dimension of the kernel and Fredholm properties, we rearrange the equations.
Substituting $\eta_{1}'$ from \eqref{etaLinearization} in \eqref{uLinearization}, and dividing by the leading coefficient, 
we arrive at
\begin{equation}\label{thirdOrder}
  u_0'''
  +\left(\frac{8\overline{\beta}}{2c^2(4\overline{\alpha}+r_0)} - \frac{8}{(4\overline{\alpha}+r_0)r_0}\right)u_0' = 0.
\end{equation}

The third-order differential equation \eqref{thirdOrder} has characteristic polynomial
\begin{equation}\label{cubicPolynomial}
    \xi^3 +
    \left(\frac{8\overline{\beta}}{2c^2(4\overline{\alpha}+r_0)} - \frac{8}{(4\overline{\alpha}+r_0)r_0}\right)\xi  = 0.
\end{equation}
We seek periodic solutions, for a fixed periodicity $M$, i.e. solutions satisfying
\begin{equation}\nonumber
     \eta(x+M,t) = \eta(x,t),\qquad
     u(x+M,t) = u(x,t),\quad \forall x.
\end{equation}
To have such periodic solutions, the cubic polynomial in \eqref{cubicPolynomial} must have two pure imaginary roots, which
we call $\pm iB$, and one real root, which we call $D$ :
\begin{equation}\label{desiredForm}
    (r-iB)(r+iB)(r-D) = r^3 -Dr^2 +B^2r -DB^{2}.
\end{equation}
With such roots, there would be two independent spatially periodic solutions of \eqref{thirdOrder},
\begin{equation*}
v_{1}=\cos(Bx), \qquad v_{2}=\sin(Bx).
\end{equation*}
Compatibility of our spatial period, $M,$ and the wavelength require the existence of $k\in\mathbb{N}$ such that
\begin{equation}\label{MBEquation}
    \frac{2\pi k}{B} = M.
\end{equation}

We now address the question of whether the roots of the cubic polynomial are in the desired form.
We match coefficients in \eqref{desiredForm} to the coefficients of the cubic polynomial on the right-hand side of 
\eqref{cubicPolynomial} so that we can find any restrictions on the parameters. Notice that \eqref{cubicPolynomial} doesn't have an 
$r^2$ term, so $D = 0.$ Continuing, we let
\begin{equation}\label{solvedForB2}
        B^2 
    = \left(\frac{8\overline{\beta}}{2c^2(4\overline{\alpha}+r_0)} - \frac{8}{(4\overline{\alpha}+r_0)r_0}\right).
\end{equation}
We can see that we may take $B$ to be real (and positive) if
\begin{equation}\nonumber
 c^2 \ < \ \frac{\overline{\beta}r_0}{2}.
\end{equation}
Let $M$ be given; then, if $c$ satisfies \eqref{solvedForB2}, to also satisfy \eqref{MBEquation}, $c$ must be given by
\begin{equation}\label{cDefinition}
c^2 = \frac{8\overline{\beta}r_0M^2}{16M^2 +8\pi^{2} k^2(4\overline{\alpha}+r_0)r_0}.
\end{equation}
That is, given a choice of $M$, there are infinitely many values $c$ (one corresponding to each $k \in \mathbb{N}$), 
which give a nontrivial periodic kernel for $L.$  Henceforth, we will take $M=2\pi.$

Moreover, here we can define the linear operator $\partial_{(\eta,u)}\phi((\eta_0,u_0),c)$:
\begin{equation*}
    L(c)(\eta_0,u_0) = 
    \begin{pmatrix}
    \partial_x && -\frac{r_0}{2c} \partial_x \\
    0 && \partial^3_x + B^2 \partial_x 
    \end{pmatrix}
    \begin{pmatrix}
     \eta_0\\u_0
    \end{pmatrix}
\end{equation*}

Given any $k_{0}\in\mathbb{N},$ with $k_{0}\neq 0,$ with the choice $M=2\pi,$ the formula \eqref{cDefinition} for $c$ becomes
\begin{equation*}
    c_0 = \left(\frac{4\overline{\beta}r_0}{k_0^2r_0(4\overline{\alpha}+r_0)+8}\right)^{\frac{1}{2}}.
\end{equation*}
Then, the possible kernel functions of $L(c_0)$ are:
\begin{equation*}
  z_1 = \cos(k_0x) + i\sin(k_0x) \ \text{and} \ 
  z_2 = \cos(k_0x) - i\sin(k_0x).
\end{equation*}
Recall that we are considering the domain $\mathcal{H}': H^1_{e,0}\times H^3_{e,0};$ therefore, we may eliminate odd functions 
from the kernel.  As a result, we have a one-dimensional kernel of $L(c_0):$
\begin{equation*}
    \text{ker}L(c_0) = 
    \left\{
    \text{span}\{h'\}|\ h' = \begin{pmatrix}
    \frac{r_0}{2c_0}\cos(k_0x)\\
    \cos(k_0x)
    \end{pmatrix} \in \mathcal{H}'
    \right\}.
\end{equation*}

Now, we need to show that the Fredholm index of the linear operator $L(c_0)$ is zero. 
For $L(c_0)$ to be Fredholm, the following must hold:
\begin{itemize}
\item ker($L(c_0)$) is finite dimensional,
\item coker($L(c_0)$) is finite dimensional, and
\item Range($L(c_0)$) is closed.
\end{itemize}
If $L(c_{0})$ is Fredholm, the index of $L(c_{0})$ is
\begin{equation*}
    \text{Ind}(L(c_0)) = \text{dim}(\text{ker} L(c_0)) - \text{dim}(\text{coker} L(c_0)).
\end{equation*}
Notice we have already shown the first condition: the dimension of the kernel of $L(c_0)$ is one-dimensional.
We will show the dimension of the kernel of $L(c_0)$ is same as the dimension of the cokernel of $L(c_0)$, so Ind$(L(c_0))$ 
becomes zero. 

We begin by demonstrating that the kernel of the adjoint of $L(c_0)$ is one-dimensional. 
The adjoint of $L(c_{0})$ is given by
\begin{equation*}
   L^*(c_0)(\eta_{1},u_{1})
    = \begin{pmatrix}
    -\eta_1' \\
    \frac{r_0}{2c_0}\eta_1'- u_1''' 
  -{k_0}^2u_1'
    \end{pmatrix}.
\end{equation*}
Now, we will find kernel of $L^*,$ as a subset of $\mathcal{H}: L^2_{odd} \times L^2_{odd}.$ 
Notice when $\eta'_1 = 0$, we have $\eta_{1}=0.$  Then the equation $\frac{r_{0}}{2c_{0}}-u_1''' - {k_0}^2 u_1' = 0$
becomes $-u_1''' - {k_0}^2 u_1' = 0.$
This is a third-order linear differential equation with characteristic polynomial $-r^3-{k_0}^2r = 0,$ and the roots of this 
are $r=0$ and $r = \pm i{k_0}$. On our domain $\mathcal{H},$ then, the kernel becomes
\begin{equation*}
    \text{ker}L^*(c_0) 
    =\left\{ \text{span}\{h^*\} \ |\  h^* = \begin{pmatrix}
   0 \\ \sin(k_0x)
    \end{pmatrix} \in \mathcal{H}
    \right\}.
\end{equation*}
Hence, $\text{ker}(L^{*}(c_0))$ is one-dimensional. It remains to establish that this is the same as the dimension of the cokernel.

Now, we move on to prove that the range of $L(c_0)$ is closed. Recall that $L(c_0)$ maps $\mathcal{H}'$ to $\mathcal{H}$. 
We introduce the decompositions $\mathcal{H}' = X_{k_0}\oplus \tilde{X}$ and $\mathcal{H} = Y_{k_0} \oplus \tilde{Y}$, where 
$X_{k_{0}}$ and $Y_{k_{0}}$ are given by
\begin{equation}\nonumber
X_{k_{0}}=\mathrm{span}\left\{
\left(\begin{array}{c}\cos(k_{0}x)\\0\end{array}\right),
\left(\begin{array}{c}0\\ \cos(k_{0}x)\end{array}\right)
\right\},
\end{equation}
\begin{equation}\nonumber
Y_{k_{0}}=\mathrm{span}\left\{
\left(\begin{array}{c}\sin(k_{0}x)\\0\end{array}\right),
\left(\begin{array}{c}0\\ \sin(k_{0}x)\end{array}\right)
\right\},
\end{equation}
and $\tilde{X}$ and $\tilde{Y}$ are the complementary
subspaces.  The operator $L(c_0)$ maps  $X_{k_0}$ to $Y_{k_0}$ and maps $\tilde{X}$ to $\tilde{Y}$. 
We will now show  $L(c_0)\upharpoonright \tilde{X}$ is bijective. 

Since $\text{ker}(L(c_{0}))\subseteq X_{k_{0}},$ we see that $L(c_0)\upharpoonright \tilde{X}$ has only trivial kernel and is thus 
injective. We need to show it is surjective as well. 
That is, we will show 
\begin{equation}\label{toShowSurjective}
    \forall y \in \tilde{Y}, \ \exists x \in \tilde{X} \ \text{such that} \ L(c_0)x = y.
\end{equation}
To do so, we consider the inverse of $\hat{L}(c_0),$
\begin{equation*}
 \hat{\Gamma}(k)
 = \frac{1}{k^4-k^2{k_0}^2}
    \begin{pmatrix}
     -i\abs{k}^3 + i\abs{k}{k_0}^2&&\frac{r_0}{2c_0}\cdot i\abs{k}\\
     0 && i\abs{k}
    \end{pmatrix}.
\end{equation*}
Notice the denominator is nonzero when $k \neq 0$ and  $k\neq k_0;$ the remaining wavenumbers correspond to $\tilde{X}.$
To show \eqref{toShowSurjective},
we will demonstrate
\begin{equation*}
    \forall y\in\tilde{Y},\  \exists \ x\in\tilde{X} \ \text{s.t} \ {\Gamma}(y) = x.
\end{equation*}
That is, we will demonstrate that $\Gamma[\tilde{Y}]=\tilde{X}.$

We consider the four components of the operator as $\hat{\Gamma}_{ij}$, where $i,j \in \{1,2\},$
\begin{eqnarray*}
\hat{\Gamma}_{1,1}(k) 
&&= \frac{1}{i\abs{k}} ,
\\\hat{\Gamma}_{1,2}(k)
&&= \frac{-i\frac{r_0}{2c_0} }{\abs{k}^3 -\abs{k}k_{0}^{2}},
\\\hat{\Gamma}_{2,1}(k)
&&= 0,
\\\hat{\Gamma}_{2,2}(k)
&&= \frac{-i}{\abs{k}^3 -\abs{k}k_{0}^{2}}.
\end{eqnarray*}
Notice $\hat{\Gamma}_{1,1}$ is smoothing by one derivative since its symbol behaves like $k^{-1}$ when $k\rightarrow\infty,$ 
while $\hat{\Gamma}_{1,2}$ and 
$\hat{\Gamma}_{2,2}$  are smoothing by three derivatives since their symbols behave like $k^{-3}$ as $k\rightarrow\infty.$  
Of course, 
$\hat{\Gamma}_{2,1}$ is simply the zero operator. 
Furthermore, all four symbols are pure imaginary, and thus map odd functions to even functions. 
Thus, we have that the $\hat{\Gamma}_{ij}$ are bounded linear operators between the following spaces:
\begin{eqnarray*}
\hat{\Gamma}_{1,1}:  L^2_{odd} \rightarrow \mathcal{H}^1_{e,0},
\\
\hat{\Gamma}_{1,2} : L^2_{odd} \rightarrow \mathcal{H}^1_{e,0},
\\
\hat{\Gamma}_{2,1}: L^2_{odd} \rightarrow \mathcal{H}^3_{e,0},
\\
\hat{\Gamma}_{2,2}: L^2_{odd} \rightarrow \mathcal{H}^3_{e,0}.
\end{eqnarray*}
This implies that ${\Gamma}$ is a bounded  mapping from $\tilde{Y}\rightarrow \tilde{X}$. The existence of this inverse implies
$L(c_0)[\tilde{X}]=\tilde{Y}.$

The range of $L(c_{0})$ is therefore equal to $L(c_{0})[X_{k_{0}}]\oplus\tilde{Y}.$  Since $\tilde{Y}$ is closed and
$X_{k_{0}}$ is finite-dimensional, we conclude that the range of $L(c_{0})$ is closed.  This also implies that the dimension of
the cokernel of $L(c_{0})$ is equal to the dimension of $\mathrm{ker}(L^{*}(c_{0})).$  Therefore we have demonstrated
that $L(c_{0})$ is Fredholm, with Fredholm index zero.  We have now established {\bf{(B3)}}.

\subsection{Establishing {\bf (B4)}}\label{b4Section} 
Next, we will show that $\left<h^*,\partial_c L(c_0) h' \right>_{\mathcal{H}}$ is nonzero.
Taking the derivative of $L$ with respect to $c,$  and evaluating at $c_{0},$ we have: 
\begin{equation*}
    \partial_c L(c_0)
    =\begin{pmatrix}
   0 && \frac{r_0}{2c_0^2} \partial_x \\
    0 &&  \frac{-8\overline{\beta}}{c_0^3(4\overline{\alpha}+r_0)}\partial_x 
    \end{pmatrix}.
\end{equation*}
Then we apply this to $h',$ computing
\begin{equation}\nonumber
    \partial_c L(c_0) h'
    = \begin{pmatrix}
    0 && \frac{r_0}{2c_0^2} \partial_x \\
    0 && \frac{-8\overline{\beta}}{c_0^3(4\overline{\alpha}+r_0)}\partial_x 
    \end{pmatrix}
   \begin{pmatrix}
    \frac{r_0}{2c_0}\cos({k_0}x)\\
    \cos({k_0}x)
    \end{pmatrix}
    =
    \begin{pmatrix}
     -\frac{r_{0}k_{0}}{2c_{0}^{2}}\sin(k_{0}x)\\
     \frac{8\overline{\beta}k_0}{(4\overline{\alpha}+r_0)c_0^3}\sin(k_0x)
    \end{pmatrix}.
\end{equation}
Then, taking the inner product with $h^{*},$ we find
\begin{multline}\nonumber
    \left<h^*,\partial_c L(c_0) h' \right>_{\mathcal{H}}
    =
    \left< 
    \begin{pmatrix}
     0\\ \sin(k_0x)
    \end{pmatrix}
    ,
    \begin{pmatrix}
     -\frac{r_{0}k_{0}}{2c_{0}^{2}}\sin(k_{0}x)\\
     \frac{8\overline{\beta}k_0}{(4\overline{\alpha}+r_0)c_0^3}\sin(k_0x)
    \end{pmatrix}
    \right>_{\mathcal{H}}
\\
=
    \frac{8\overline{\beta}k_0}{(4\overline{\alpha}+r_0)c_0^3} \int_{0}^{M}\sin^2(k_0x)dx \ \neq  \ 0,
\end{multline}
 with this being nonzero since $\overline{\beta}$ and $k_{0}$ are nonzero.  
Hence, we have proved {\bf (B4)} and we may apply Theorem \ref{bifurcationTheorem},  completing the proof of Theorem 
\ref{travelingWaveExistenceTheorem}.

\section{Discussion}\label{discussionSection}
We mention here a few future directions for this line of analysis.  First, while we have given an argument that the system 
\eqref{originalEtaEquation}, \eqref{originalUEquation} has an ill-posed initial value problem when $r_{0}$ is non-constant,
we have not rigorously demonstrated ill-posedness.  Generally speaking, ill-posedness can be more challenging to prove than
well-posedness, and most often this is approached by demonstrating lack of continuous dependence on the initial data.  
Now that we have demonstrated a family of solutions for the general problem in Section \ref{ACKSection}, it is possible
that a detailed analysis of these solutions could yield insight into a lack of continuous dependence on the data.

There are of course also future directions regarding traveling waves.
We have proved the existence of periodic traveling waves in the case $\kappa=\gamma=0.$
This restriction on the parameters leads to a one-dimensional kernel of the
linearized operator, which is a hypothesis of Theorem \ref{bifurcationTheorem}. 
In the general case, the kernel is two-dimensional.  We have considered applying one-dimensional 
bifurcation theorems with two-dimensional kernels such as \cite{twodim2}, \cite{twodim1}, but have found that the 
conditions of these theorems are not satisfied by our system.  
Considering a genuinely two-dimensional bifurcation will be the subject of future work.

Additionally, there are other models of fluid flow in viscoelastic vessels, such as the work of \cite{bertagliaCaleffiValiani}.
Analysis of further models, and comparison of features of solutions across different models, is another direction for future work.
Asymptotic models such as the system \eqref{etaEquation}, \eqref{uEquation} should also be validated, in the sense that 
it should be proved that solutions of the model equation and solutions of the full equations (i.e. the Navier-Stokes equations)
remain close, if they begin with the appropriately scaled, equivalent initial data.  There is a long history of validation results
for model equations in free-surface fluid dynamics \cite{lannesBook}, and extending such results to the present setting will
be valuable to understand the sense in which these models are indeed a good approximation to the phenomena under 
consideration.

\section*{Acknowledgement}.
The authors are grateful to the National Science Foundation for support through grant DMS-1907684, to the second author.

\bibliography{main.bib}{}
\bibliographystyle{plain}

\end{document}